\documentclass[12pt,twoside,a4paper]{article}
\usepackage{amsmath,amssymb,amsthm}
\usepackage{hyperref}
\usepackage{graphicx,psfrag}
\usepackage[utf8]{inputenc}

\newtheorem{thm}{Theorem}[section]
\newtheorem{lem}[thm]{Lemma}
\newtheorem{pro}[thm]{Proposition}
\newtheorem{cor}[thm]{Corollary}

\newtheorem{definition}[thm]{Definition}

\theoremstyle{definition}

\newtheorem{exas}[thm]{Examples}

\theoremstyle{remark}
\newtheorem{rem}[thm]{Remark}

\newcommand{\R}{\mathbb{R}}
\newcommand{\Z}{\mathbb{Z}}
\newcommand{\N}{\mathbb{N}}

\newcommand{\cA}{\mathcal{A}}
\newcommand{\cB}{\mathcal{B}}
\newcommand{\cC}{\mathcal{C}}

\newcommand{\cI}{\mathcal{I}}
\newcommand{\cH}{\mathcal{H}}

\newcommand{\cM}{\mathcal{M}}

\newcommand{\cP}{\mathcal{P}}

\newcommand{\cT}{\mathcal{T}}

\newcommand{\al}{\alpha}
\newcommand{\be}{\beta}
\newcommand{\ga}{\gamma}
\newcommand{\Ga}{\Gamma}
\newcommand{\de}{\delta}
\newcommand{\De}{\Delta}
\newcommand{\ep}{\varepsilon}
\newcommand{\om}{\omega}

\newcommand{\si}{\sigma}

\newcommand{\la}{\lambda}

\renewcommand{\phi}{\varphi}

\newcommand{\dist}{\operatorname{dist}}

\newcommand{\CAT}{\operatorname{CAT}}
\newcommand{\hyp}{\operatorname{H}}

\newcommand{\ds}{\operatorname{dS}}
\newcommand{\ads}{\operatorname{AdS}}
\newcommand{\id}{\operatorname{id}}

\newcommand{\pr}{\operatorname{pr}}

\newcommand{\crr}{\operatorname{cr}}
\newcommand{\cd}{\operatorname{cd}}

\newcommand{\intr}{\operatorname{int}}

\newcommand{\reg}{\operatorname{reg}}
\renewcommand{\d}{\partial}
\newcommand{\di}{\d_{\infty}}
\newcommand{\ay}{\operatorname{aY}}
\newcommand{\rp}{\R\!\operatorname{P}}
\newcommand{\SL}{\operatorname{SL}}

\newcommand{\PSL}{\operatorname{PSL}}
\newcommand{\tr}{\operatorname{Tr}}
\newcommand{\sign}{\operatorname{sign}}

\newcommand{\co}{\operatorname{co}}

\renewcommand{\d}{\partial}
\renewcommand{\di}{\d_{\infty}}
\newcommand{\set}[2]{\{#1:\,\text{#2}\}}
\newcommand{\sm}{\setminus}
\newcommand{\sub}{\subset}

\newcommand{\ov}{\overline}
\newcommand{\wt}{\widetilde}
\newcommand{\wh}{\widehat}

\begin{document}

\title{M\"obius structures and timed causal spaces on the circle}

\author{Sergei Buyalo\footnote{This work is supported by RFBR Grant
17-01-00128a}}

\date{}
\maketitle

\begin{abstract} We discuss a conjectural duality between
hyperbolic spaces on one hand and spacetimes on the other hand,
living on the opposite sides of the common absolute. This
duality goes via M\"obius structures on the absolute,
and it is easily recognized in the classical case of symmetric
rank one spaces. In a general case, no trace of such duality
is known. As a first step in this direction, we show how M\"obius 
structures on the circle from a large class including those which 
stem from hyperbolic spaces give rise to 2-dimensional spacetimes, 
which are axiomatic versions of de Sitter 2-space, and vice versa.
The paper has two Appendices, one of which is written by V.~Schroeder.
\end{abstract}

\noindent{\small{\bf Keywords:} M\"obius structures, cross-ratio, harmonic 4-tuples,
hyperbolic spaces, spacetimes, de Sitter space}

\medskip

\noindent{\small{\bf Mathematics Subject Classification:} 51B10, 53C50}

\section{Introduction}

It is classical that the quadratic form 
$$g(v)=x^2+y^2-z^2$$
on
$\R^3$, $v=(x,y,z)\in\R^3$,
induces on any connected component of the set 
$g(v)=-1$
a Riemannian metric of the hyperbolic plane
$\hyp^2$,
while on the set 
$g(v)=1$
a Lorentz metric of de Sitter 2-space
$\ds^2$. 
The (set of lines in the) cone
$g(v)=0$
serves as the common absolute 
$S^1$
of the both
$\hyp^2$
and
$\ds^2$.
A similar picture takes place in any dimension and even for all rank one 
symmetric spaces of noncompact type.

In other words, we observe a life on the other side of the absolute
$S^1$
of
$\hyp^2$
that is de Sitter space
$\ds^2$. 
For mathematical aspects of the duality between hyperbolic spaces
$\hyp^{n+1}$
and de Sitter spacetimes
$\ds^{n+1}$
see e.g. \cite{Ge}, \cite{Yu}. Interplay between geometry of hyperbolic surfaces and 
Lorentz (2+1)-spaces is exploited in the famous paper \cite{Mes}, see
also \cite{A-S}. Duality for quadratic forms of arbitrary signature 
is discussed in \cite{Ro}. For physical aspects of de Sitter spaces
see e.g. \cite{SSV} and references therein.

In sect.~\ref{sect:other_side}, we describe this duality
in intrinsic terms. The basic feature is the canonical M\"obius structure
$M_0$
on the absolute
$S^1$,
which governs its both sides
$\hyp^2$
and
$\ds^2$.
In particular, the isometry groups of
$\hyp^2$
and 
$\ds^2$
coincide with the group of M\"obius automorphisms of
$M_0$. 
We show how to recover the hyperbolic plane
$\hyp^2$
and de Sitter 2-space
$\ds^2$
purely out of
$M_0$.

Moreover, we explain a mechanism of the passage from
$\hyp^2$
to
$\ds^2$
and back. Shortly,
$\hyp^2$
is the homogeneous space of the 
$M_0$-automorphism 
group
$\PSL_2(\R)$
over a compact elliptic subgroup isomorphic to
$S^1$,
while
$\ds^2$
is the homogeneous space of
$\PSL_2(\R)$
over a (closed) hyperbolic subgroup isomorphic to
$\R$.

This rises a bold question: Is there any life (a spacetime) on the other side
of the absolute, i.e., the boundary at infinity, of any Gromov hyperbolic space
with the same symmetry group? The main result of the paper is the answer 
``yes'' for a large class of hyperbolic spaces with the absolute
$S^1$, 
see Theorem~\ref{thm:main}.

A M\"obius structure on a set 
$X$
is a class of semi-metrics having one and the same cross-ratio
on any given ordered 4-tuple of distinct points in
$X$,
see sect.~\ref{sect:moeb}. Every hyperbolic space
$Y$
induces on its boundary at infinity
$X=\di Y$
a M\"obius structure which encodes most essential properties of
$Y$
and in a number of cases allows to recover
$Y$
completely, e.g., in the case
$Y$
is a rank one symmetric space of noncompact type, see
\cite{BS2}, \cite{BS3}. In sect.~\ref{subsect:boundary_continuous},
we explain this for the class of {\em boundary continuous} hyperbolic
spaces. In Appendix~1, sect.~\ref{sect:appendix_1},  we show that every proper
Gromov hyperbolic
$\CAT(0)$
space is boundary continuous.

We axiomatically describe a class
$\cM$
of {\em monotone} M\"obius structures on the circle
$S^1$,
see sect.~\ref{sect:monotone}. The class 
$\cM$
includes every M\"obius structure
$M$
which is induced on
$S^1$
by a hyperbolic 
$\CAT(0)$ 
surface 
$Y$
without singular points, see Theorem~\ref{thm:without_singular}.
In particular, the isometry group of
$Y$
is included in the group of M\"obius automorphisms of
$M$.
Furthermore, the canonical
M\"obius structure
$M_0$
is the most symmetric representative from
$\cM$.

On the other hand, the set 
$\ay$
of unordered pairs of distinct points on the circle
$X=S^1$
has a natural causal structure, which is independent of anything else, see
sect.~\ref{sect:other_side}. Points of
$\ay$
are called {\em events}. There is a large class 
$\cT$
of 2-dimensional spacetimes compatible with that causal structure,
and we characterize it axiomatically in sect.~\ref{sect:timed_causal}. 
Any spacetime
$T\in\cT$
is a triple
$T=(\ay,\cH,t)$,
where 
$\cH$
is a class of timelike curves in
$\ay$
called {\em timelike lines}, which are actually timelike geodesics,
and
$t$
is the time between events in the causal relation. The spacetime
$T\in\cT$
is called {\em timed causal space}.
We prove

\begin{thm}\label{thm:main} There are natural mutually inverse maps 
$\wh T:\cM\to\cT$
and 
$\wh M:\cT\to\cM$
such that the groups of automorphisms of any 
$M\in\cM$
and of the respective
$T=\wh T(M)\in\cT$
are canonically isomorphic.
\end{thm}

It follows from constructions of sect.~\ref{sect:other_side} that 
the canonical M\"obius structure
$M_0$
on
$S^1$
determines de Sitter space
$\ds^2$,
that is,
$\wh T(M_0)=\ds^2$
and 
$\wh M(\ds^2)=M_0$.
In other words, Theorem~\ref{thm:main} says that a monotone M\"obius structure on 
$S^1$
on one hand, and the respective timed causal space with the absolute
$S^1$ 
on the other hand, are different sides of one and the same phenomenon
also in a general case.

The fundamental feature of spacetimes is the {\em time inequality.}
In section~\ref{sect:time_inequality}, we discuss a hierarchy
of time conditions, in particular, we introduce the {\em weak time inequality,}
and show that every timed causal space
$T\in\cT$
satisfies the weak time inequality, see Theorem~\ref{thm:wti}. 

In sect.~\ref{subsect:monotone_vp}, we introduce Increment Axiom~(I)
which implies the time inequality, and show that the subset
$\cI\sub\cM$
of M\"obius structures satisfying (I) contains the canonical structure
$M_0$, $M_0\in\cI$,
(Proposition~\ref{pro:canonical_vp}) with a neighborhood of
$M_0$
in the fine topology (Proposition~\ref{pro:pertubed_canonical_vp}).

In sect.~\ref{subsect:convex_moeb}, we introduce Convexity Axiom~(C)
for monotone M\"obius structures
$M\in\cM$,
which implies convexity of a functional 
$F_{ab}$
playing an important role in the hierarchy of time conditions and 
show that the subset
$\cC\sub\cM$
of convex M\"obius structures contains
$M_0$
(Proposition~\ref{pro:canonical_convex}).

The spacetimes of the class
$\cT$
are related to de Sitter 2-space
$\ds^2$
in a sense at least as hyperbolic
$\CAT(0)$
surfaces without singular points with the absolute
$S^1$ 
are related to the hyperbolic plane
$\hyp^2$.
If one would extend results of this paper to more general hyperbolic
spaces even with 1-dimension boundary at infinity, then this potentially
could produce new interesting classes of spacetimes e.g. having
a branching time (timelike lines).

\bigskip
{\em Acknowledgments.} The author is very much thankful to 
Prof.Dr. Viktor Schroeder for attention to this paper and valuable remarks. 
Especially, for pointing me out that Axiom~(t6) from Sect.~\ref{subsect:time}
follows from the other axioms of timed causal spaces. This is explained
in details in Appendix~2 written by Viktor Schroeder.

\tableofcontents

\section{On the other side of the absolute}
\label{sect:other_side}

This section serves as a motivation, contains no new result, and its constructions
are widely known. Here, we show how the both sides
$\hyp^2$
and
$\ds^2$
of the common absolute
$S^1$
can be recovered out of the canonical M\"obius structure 
$M_0$
on
$S^1$.
It is common to define 
$\hyp^2$
and
$\ds^2$
in the quotient of
$\R^3$
by the antipodal map 
$x\mapsto -x$.
This does not affect
$\hyp^2$,
while
$\ds^2$
becomes a nontrivial line bundle over
$S^1$,
that is, the open M\"obius band.

\subsection{Recovering the hyperbolic plane $\hyp^2$}
\label{subsect:recovering_hyp}

The canonical M\"obius structure 
$M_0$
on the circle
$S^1$
is determined by the condition that any its representative with
infinitely remote point is a standard metric (up to a positive factor) on
$\wh\R=\R\cup\{\infty\}$
extended in the sense that the distance between any 
$x\in\R$
and
$\infty$
is infinite. 

To recover
$\hyp^2$
from
$M_0$,
we consider the space 
$Y$
of all M\"obius involutions
$s:S^1\to S^1$
with respect to
$M_0$
without fixed points. The space 
$Y$
serves as the underlying space for
$\hyp^2$,
and what remains to do is to introduce a respective metric on
$Y$.

A line in
$Y$
is determined by a pair 
$x$, $x'\in S^1$
of distinct points and consists of all involutions
$s\in Y$
which permute
$x$, $x'$, $sx=x'$.
Given two distinct points
$s$, $s'\in Y$,
the compositions
$s's$, $ss':S^1\to S^1$
have one and the same fixed point set consisting of
two distinct points
$x$, $x'\in S^1$.
Thus there is a uniquely determined line in
$Y$
through
$s$, $s'$.

We say that an ordered 4-tuple 
$q=(x,x',y,z)\in(S^1)^4$
of pairwise distinct points is {\em harmonic} if 
\begin{equation}\label{eq:harmonic}
|xy|\cdot|x'z|=|xz|\cdot|x'y| 
\end{equation}
for some and hence any metric on
$S^1$
from
$M_0$.
For the canonical M\"obius structure
$M_0$,
harmonicity of
$(x,x',y,z)$
is equivalent to that the geodesic lines
$xx'$, $yz\sub\hyp^2$
are mutually orthogonal.

A {\em sphere}
$S$
between
$x$, $x'\in S^1$
is of a pair
$(y,z)\sub S^1$
such that the 4-tuple
$(x,x',y,z)$
is harmonic. We take spheres
$S$, $S'\sub S^1$
between
$x$, $x'$
such that 
$S$
is invariant under 
$s$, $s(S)=S$,
and
$S'$
is invariant under
$s'$, $s'(S')=S'$.
The spheres 
$S$, $S'$
with these properties exist and are uniquely determined. Now, we
take
$y\in S$, $y'\in S'$
and put
\begin{equation}\label{eq:distance_log}
|ss'|=|\ln\langle x,y,y',x'\rangle|,
\end{equation}
where
$\langle x,y,y',x'\rangle=\frac{|xy'|\cdot|yx'|}{|xy|\cdot|y'x'|}$
is the cross-ratio of the 4-tuple
$(x,y,y',x')$.
This is well defined and independent of the choice
$y\in S$, $y'\in S'$.
It is easy to show that
$|ss'|$
is the distance in the geometry of
$\hyp^2$,
see sect.~\ref{subsect:auto}.

\begin{rem}\label{rem:rank_one_extension} This construction is easily
extended to any rank one symmetric space of noncompact type,
see \cite{BS2}.
\end{rem}

\subsection{Recovering de Sitter space $\ds^2$}
\label{subsect:recoveting_desitter}

Let
$\ay$
be the space of unordered pairs 
$(x,y)\sim(y,x)$
of distinct points on
$S^1$
with the induced from
$S^1$
topology, that is,
$\ay=S^1\times S^1\sm\De/\sim$,
where 
$\De=\set{(x,x)}{$x\in S^1$}$
is the diagonal. Then 
$\ay$
is a nontrivial 
$\R$-bundle
over
$\rp^1\approx S^1$,
i.e., 
$\ay$
is the open M\"obius band. In this case, 
$S^1$
is the boundary of
$\ay$
at infinity,
$\di\ay=S^1$.
Points of
$\ay$
are called {\em events.}

We say that two events
$e$, $e'\in\ay$
are in the {\em causal relation} if and only if
$e$, $e'$
do not separate each other as pairs of points in
$S^1$.
This defines the {\em canonical causality structure} on
$\ay$.

A {\em light line} in
$\ay$
is determined by any
$x\in S^1$
and consists of all events
$a=(x,x')\in\ay$, $x'\in S^1\sm x$.
For this light line
$p_x$,
$x$
is the unique point at infinity. Two distinct light lines
$p_x$, $p_y$
have a unique common event 
$(x,y)\in\ay$,
and any two events on a light line are in the causal relation.

The canonical causality structure as well as light lines are inherent in
$\ay$,
and they do not depend of anything else.

\begin{rem}\label{rem:high_dim} In higher dimensional case, a
causality structure can be defined similarly, but then it
depends on the M\"obius structure because events are codimension one 
spheres in
$S^n$.
\end{rem}

A {\em timelike line} in
$\ay$
is determined by any event
$e\in\ay$
and consists of all
$a\in\ay$
such that the 4-tuple
$(e,a)$
is harmonic. For the timelike line
$h_e\sub\ay$
determined by
$e=(x,x')$,
the points
$x$, $x'\in S^1$
are the ends of
$h_e$
at infinity. It follows from definitions that 
$a\in h_e$
if and only if
$e\in h_a$.

Any two events on a timelike line are in the causal relation.
Conversely, for any two events
$a$, $a'\in\ay$
which are in the causal relation and not on a light line there is 
a unique timelike line (the common perpendicular)
$h_e$
with
$a$, $a'\in h_e$
(this amounts to existence and uniqueness of a common perpendicular
to divergent geodesics in
$\hyp^2$.
For a (de Sitter) proof see Corollary~\ref{cor:common_perpendicular_prescribed}
and Lemma~\ref{lem:unique_common_perpendicular}).
Let 
$a=(y,z)$, $a'=(y',z')$, $e=(x,x')$
in this case. Then the {\em time}
$t=t(a,a')$
between the events
$a$, $a'$
is defined by formula (\ref{eq:distance_log}) 
$$t=|\ln\langle x,y,y',x'\rangle|$$
(note that
$a$, $a'$
are sphere between
$x$, $x'$).

It follows that two timelike lines
$h_e$, $h_{e'}$
intersect each other if and only if the events
$e$, $e'\in\ay$
are in the causal relation and not on a light line. In this case, the intersection
$h_e\cap h_{e'}$
is a unique event.

An {\em elliptic line} in
$\ay$
is determined by any M\"obius involution without fixed points
$s\in Y$
and consists of all events
$a\in\ay$
such that
$sa=a$.
No two distinct events on an elliptic line are in the causal
relation.

\begin{rem}\label{rem:existence_elliptic} The last definition make sense
only for the canonical M\"obius structure
$M_0$
because in a general case a M\"obius structure may not admit any M\"obius
involution without fixed points.
\end{rem}

\subsection{Automorphisms of $M_0$}
\label{subsect:auto}

To introduce a metric structure on
$\ay$
we consider the Lie algebra
$\mathfrak g$
of the Lie group
$G=\SL_2(\R)$.
Given
$\al$, $\be\in\mathfrak g$
we have the Killing form
\begin{equation}\label{eq:killing_form}
\langle\al,\be\rangle=\frac{1}{2}\tr(\al\be) 
\end{equation}
as a scalar product. Note that the matrices
$\si_1,\si_2,\si_3\in\mathfrak g$,
$$\si_1= 
\begin{bmatrix}
 1&0\\
 0&-1
\end{bmatrix},\quad
\si_2= 
\begin{bmatrix}
 0&1\\
 1&0
\end{bmatrix},\quad
\si_3= 
\begin{bmatrix}
 0&1\\
 -1&0
\end{bmatrix},
$$
are mutually orthogonal and
$\|\si_1\|^2=\langle\si_1,\si_1\rangle=1=\|\si_2\|^2$,
$\|\si_3\|^2=-1$.

The group
$G$
acts on
$\wh\R=\R\cup\{\infty\}$
by linear-fractional transformations
\[x\mapsto\frac{ax+b}{cx+d}\quad\text{for}\quad\begin{bmatrix}
   a&b\\
   c&d
  \end{bmatrix}
\in G\]
which are M\"obius with respect to the canonical M\"obius structure
$M_0$.
The action is not effective with the kernel
$\Z_2=\{\pm\id\}$, $G/\Z_2=\PSL_2(\R)$.
The group
$\PSL_2(\R)$
with the left invariant Lorentz metric (\ref{eq:killing_form})
is anti de Sitter 3-space
$\ads^3$.

We denote by
$K_i=\set{\exp(t\si_i)}{$t\in\R$}$, $i=1,2,3$,
a 1-parametric subgroup in
$G$,
and by
$\wh K_i$
its image in
$\PSL_2(\R)$.
Note that
$\wh K_i=K_i$
for
$i=1,2$
and
that
$\wh K_3=K_3/\Z_2$.

The space
$Y$
of M\"obius involutions
$s:S^1\to S^1$
without fixed points can be identified with the homogeneous space
$G/K_3=\PSL_2(\R)/\wh K_3$
because the group
$$K_3=\set{g_3(t)=\exp(t\si_3)=\begin{bmatrix}
                \cos t&\sin t\\
               -\sin t&\cos t
               \end{bmatrix}}{$t\in\R$}$$
stabilizes
$s=g_3(\frac{\pi}{2})=\begin{bmatrix}
                           0&1\\
                          -1&0 
                          \end{bmatrix}$
which acts on
$\wh\R$
as the M\"obius involution
$s(x)=-\frac{1}{x}$
without fixed points. The space
$G/K_3$
carries a left-invariant Riemannian metric 
$h_3$
originated from the subspace
$L_3\sub\mathfrak g$
spanned by
$\si_1$, $\si_2$,
and
$(G/K_3,h_3)$
is isometric to
$\hyp^2$.
To see that we compute the respective Riemannian distance
between two involutions
$s_1$, $s_2\in Y$.
By conjugation we can assume that
$s_1=s$, $s_2=s'$,
where
$s'=g_1(t)\cdot s\cdot g_1^{-1}(t)$
for some
$t\in\R$,
$$g_1(t)=\exp(t\si_1)=\begin{bmatrix}
                       e^t&0\\
                         0&e^{-t}.
                      \end{bmatrix}
$$
Then
$s'=\begin{bmatrix}
     0&e^{2t}\\
    -e^{-2t}&0
    \end{bmatrix}$
and
$s'(x)=-\frac{e^{4t}}{x}$.
The curve
$t\mapsto g_1(t)$
is a unit speed geodesic in
$G$.
While projected to
$\PSL_2(\R)$
the speed is doubled because of linear-fractional action of
$\PSL_2(\R)$, 
so we have
$|ss'|=2t$.
In the upper half-plane model of
$\hyp^2$
the involution
$s$
fixes
$i=(0,1)$
with Euclidean distance
$|0i|_e=1$
and
$s'$
fixes
$ie^{2t}=(0,e^{2t})$
with Euclidean distance
$|0ie^{2t}|_e=e^{2t}$,
hence
$|ss'|=2t$
equals the 
$\hyp^2$-distance
$|(0,1)(0,e^{2t})|=\ln\frac{|0ie^{2t}|_e}{|0i|_e}=2t$.

The action of
$\PSL_2(\R)$
on
$\wh\R\approx S^1$
induces the {\em standard} action of
$\PSL_2(\R)$
on
$\ay$.
Note that 
$a=\{-1,1\}\in\ay$
is a fixed point for 
$K_2$
because
$$\exp(t\si_2)=\begin{bmatrix}
                \cosh t&\sinh t\\
                \sinh t&\cosh t
               \end{bmatrix},$$
and
$$\exp(t\si_2)a=\left\{\frac{\cosh t(-1)+\sinh t}{\sinh t(-1)+\cosh t},
  \frac{\cosh t+\sinh t}{\sinh t+\cosh t}\right\}=\{-1,1\}=a.$$
It follows that 
$\ay$
can be identified with the homogeneous space
$\PSL_2(\R)/K_2$, 
or similarly with
$\PSL_2(\R)/K_1$.
The space
$\ay=\PSL_2(\R)/K_2$
carries a left-invariant Lorentz metric 
$h_2$
originated from the subspace
$L_2\sub\mathfrak g$
spanned by
$\si_1$, $\si_3$,
and 
$\ds^2=(\ay,h_2)$.

In the rest of the paper, we explain how a M\"obius structure
$M$
from a large class of structures on the circle gives rise to a spacetime,
and vice versa, without any assumption on symmetries of
$M$.

\section{Timed causal spaces on the circle}
\label{sect:timed_causal}

In this section we list axioms for timed causal spaces
on the circle.

\subsection{The canonical causality structure}
\label{subsect:causality_structure}

Recall that on the space
$\ay$
of unordered pairs of distinct point in
$X=S^1$,
which is homeomorphic to the open M\"obius band, we have the canonical causal 
structure. That is, events
$e$, $e'\in\ay$
are in the causal relation if and only if they do not separate
each other as pairs of points in
$X$.
In the opposite case, we also say that events
$e$, $e'\in\ay$
separate each other. In the case
$e$, $e'\in\ay$
are in the causal relation and not on a light line, we say 
that the events
$e$, $e'$
are in the {\em strong} causal relation. 

The canonical causal structure and light lines inherent to
$\ay$,
see sect.~\ref{subsect:recoveting_desitter}. 

For a fixed event
$e\in\ay$
the set 
$C_e$
of all
$e'\in\ay$
in the causal relation with
$e$
is called the {\em causal cone}. The pair
$e\sub X$
decomposes
$X$
into two closed arcs, which we denote by
$e^+$, $e^-$,
with
$e^+\cap e^-=e$.
Every
$a\in\ay$
with
$a\sub e^\pm$
is in the causal relation with 
$e$.
We let
$$C_e^\pm=\set{a\in\ay}{$a\sub e^\pm$}.$$
Therefore, a choice of
$e^+$, $e^-$
induces the decomposition
$C_e=C_e^+\cup C_e^-$
of the causal cone
$C_e$
into the {\em future} cone
$C_e^+$
and the {\em past} cone
$C_e^-$
with
$C_e^+\cap C_e^-=e$,
and moreover introduces a partial order on
$\ay$
in the following way. Every
$a\in C_e^\pm$
decomposes
$X$
into two closed arcs, and if
$a\neq e$,
we canonically define
$a_e^\pm$
as one of them that does not contain
$e$,
otherwise
$a_e^\pm=e^\pm$.
Now we define:
$a\le_e a'$
if and only if one of the following holds

\begin{itemize}
 \item $a\in C_e^-$, $a'\in C_e^+$
 \item $a'\sub a_e^+$ 
if
$a,a'\in C_e^+$
and
$a\sub (a')_e^-$
if
$a,a'\in C_e^-$.
\end{itemize}
As usual, we say that
$a<_e a'$,
if
$a\le_e a'$
and
$a\neq a'$.

Note that there is no global partial order on
$\ay$
compatible with the canonical causal structure, and the order above
only appears if an event
$e\in\ay$,
future
$e^+$
and past 
$e^-$
arcs are chosen.

\subsection{Timelike lines and a causal space}
\label{subsect:hyplines}

The notion of a {\em timelike line} is not inherent in
$\ay$,
and we define this notion axiomatically. 

\medskip\noindent
{\bf Axioms for timelike lines}
\begin{itemize}
 \item[(h1)] every event 
$e\in\ay$
uniquely determines a timelike line
$h_e\sub\ay$,
and every timelike line in
$\ay$
is of form
$h_e$
for some
$e\in\ay$;
 \item[(h2)] any event
$a\in h_e$
separates
$e$;
 \item[(h3)] any two events on a timelike line are in the causal
relation;
 \item[(h4)] for any point
$x\in X\sm e$
there is a unique event 
$x_e=(x,y)\in h_e$;
 \item[(h5)] if an event 
$a\in\ay$
is on a timelike line
$h_e$, 
then
$e\in h_a$;
 \item[(h6)] for any two distinct events
$a$, $a'\in\ay$
there is at most one timelike line
$h_e$
with
$a$, $a'\in h_e$.
\end{itemize}

The space
$\ay$
with a fixed collection 
$\cH$
of subsets satisfying the axioms of
timelike lines is called a {\em causal space}. We use 
notation
$(\ay,\cH)$
for a causal space. In view of Axiom~(h1), we say that an event 
$e\in\ay$
and the timelike line
$h_e\sub\ay$
are {\em dual} to each other. 

It follows from (h4) that for every event
$e=(z,u)\in\ay$
we have a well defined map 
$\rho_e:X\to X$
given by
$\rho_e(z)=z$, $\rho_e(u)=u$,
and 
$(x,\rho_e(x))=x_e$
for every
$x\in X\sm e$.
The map 
$\rho_e$
is called the {\em reflection with respect to}
$e$.

\begin{lem}\label{lem:reflection} For every
$e=(z,u)\in\ay$,
the map 
$\rho_e:X\to X$
is an involutive homeomorphism.
\end{lem}

\begin{proof} By definition,
$\rho_e^2(z)=z$, $\rho_e^2(u)=u$.
Let 
$y=\rho_e(x)$
for 
$x\in X\sm e$.
Then events
$(x,y),(\rho_e(y),y)\in h_e$,
hence
$\rho_e(y)=x$
by (h4). Thus
$\rho_e^{-1}=\rho_e$,
and
$\rho_e$
is a bijection.

The event
$e$
decomposes
$X$
into two closed arcs
$e^+$, $e^-$
with
$X=e^+\cup e^-$, $e^+\cap e^-=e$.
An orientation of
$X$
determines linear orders on
$e^+$, $e^-$.
It follows from (h2) and (h3) that 
$\rho_e:e^+\to e^-$
reverses the orders. Hence,
$\rho_e$
is continuous and, therefore, a homeomorphism.
\end{proof}

\begin{pro}\label{pro:timelike_lines} Let
$(\ay,\cH)$
be a causal space. Then
\begin{itemize}
 \item[(a)] for any 
$e\in\ay$
the line 
$h_e$
is homeomorphic (in the induced from
$\ay$
topology) to
$\R$,
and the boundary of the closure
$\ov h_e\sub\ay\cap\di\ay$
is
$e$, $\d\ov h_e=e$;
 \item[(b)] for any two distinct events
$a$, $a'\in\ay$
there is a timelike line
$h$
including
$a$, $a'$
if and only if
$a$
and
$a'$
are in the strong causal relation. In this case,
$h$
is unique with this property;
 \item[(c)] any two distinct timelike lines
$h_e$, $h_{e'}\in\cH$
have a common event
$a$ 
if and only if
$e$, $e'$
are in the strong causal relation. 
In this case,
$a$
is unique;
 \item[(d)] a light line
$p_x\sub\ay$
intersects a timelike line
$h_e\sub\ay$
if and only if
$x\not\in e$.
In this case, the common event
$a\in p_x\cap h_e$
is unique.
\end{itemize}
\end{pro}

\begin{proof} (a) Let
$\rho_e:X\to X$
be the reflection with respect to
$e$, $e^+\sub X$
one of the two closed arcs, in which
$e$
decomposes
$X$.
Then the map 
$\intr e^+\to h_e$, $x\mapsto(x,\rho_e(x))$
is an order preserving bijection. Extended to
$e^+$,
it gives an order preserving bijection to
$\ov h_e=h_e\cup e$. 
Thus
$h_e$
is homeomorphic to
$\intr e^+\approx\R$,
and
$\d\ov h_e=e$.

(b) Any distinct
$a$, $a'\in\ay$
on a timelike line
$h$
are in the causal relation by (h3), and by (h4)
they are not on a light line. Hence,
$a$, $a'$
are in the strong causal relation. Conversely, assume that events
$a$, $a'\in\ay$
are in the strong causal relation. Let
$\rho=\rho_a\circ\rho_{a'}$
be the composition of respective reflections,
$a^+\sub X$
the closed arc determined by
$a$
that does not contain
$a'$.
Then
$\rho(a^+)\sub\intr a^+$,
and thus there is a fixed point
$x\in\intr a^+$
of
$\rho$, $\rho(x)=x$.
It follows that the both reflections
$\rho_a$, $\rho_{a'}$
preserve the event
$e=(x,y)$,
where
$y=\rho_{a'}(x)$.
Hence,
$e\in h_a\cap h_{a'}$,
and by (h5),
$a,a'\in h_e$.
By (h6),
$h_e$
is unique with this property.

(c) By duality (h5), this is a reformulation of (b). 

(d) This immediately follows from (h2) and (h4).
\end{proof}

\begin{rem}\label{rem:axiom_h6} Axiom~(h6) is not used in
Lemma~\ref{lem:reflection}, and it is only used in 
Proposition~\ref{pro:timelike_lines} to prove the uniqueness in (b) and (c). 
Thus all the conclusions of Proposition~\ref{pro:timelike_lines} 
except for the uniqueness in (b) and (c) hold true without Axiom~(h6). 
We use this remark in sect.~\ref{subsect:timelike_lines}.
\end{rem}

\subsection{Timed causal space}
\label{subsect:time}

The notion of a {\em time} is also defined axiomatically. 

\medskip\noindent
{\bf The time axioms}
\begin{itemize}
 \item[(t1)]
A time 
$t(e,e')\ge 0$
between two events
$e$, $e'\in\ay$
is determined if and only if
$e$, $e'$
are in the causal relation;
 \item[(t2)] $t(e,e')=0$
if and only if
$e$, $e'$
are events on a light line;
 \item[(t3)] $t(e,e')=t(e',e)$
whenever it is defined;
 \item[(t4)] timelike lines are 
$t$-geodesics: 

(a) if
$e,e',e''\in h_a$
are events on a timelike line such that 
$e\le e'\le e''$,
then
$t(e,e')+t(e',e'')=t(e,e'')$;

(b) for every
$e\in h_a$
and every
$s>0$
there are
$e_\pm\in h_a\cap C_e^\pm$
with
$t(e,e_\pm)=s$;
 \item[(t5)] for any events
$e=(x,y)$, $d=(z,u)$
in the strong causal relation, we have
$t(z_e,u_e)=t(x_d,y_d)$;
 \item[(t6)] for any
$e=(x,y), d=(z,u)\in h_e$
the 4-tuple
$(d,e)$
is {\em harmonic} in the sense that 
$t(y_a,u_a)=t(y_b,z_b)$,
where
$a=(x,z)$, $b=(x,u)$.
\end{itemize}

A {\em timed causal space} is defined as
$T=(\ay,\cH,t)$,
where
$t$
is a time on the causal space
$(\ay,\cH)$.
This is a version of Busemann (locally) timelike spaces,
see \cite{Bus}, and also an axiomatic version of the de Sitter space
$\ds^2$.
Since 
$\ds^2$
is recovered from the canonical M\"obius structure 
$M_0$
on the circle, see sect.~\ref{subsect:recoveting_desitter},
it follows from results of sect.~\ref{sect:monotone} (see
Proposition~\ref{pro:safisfies_hypline_axiom}, 
Lemma~\ref{lem:unique_common_perpendicular} and 
Proposition~\ref{pro:satisfies_time_axioms}), that 
$\ds^2$
is a timed causal space. 

We denote by
$\cT$
the set of all timed causal spaces
$(\ay,\cH,t)$,
where the collection
$\cH$
of timelike line satisfies Axioms~(h1)--(h6), and the time
$t$
satisfies Axioms~(t1)--(t6). A
$T$-{\em automorphism}, $T=(\ay,\cH,t)\in\cT$,
is a bijection
$g:\ay\to\ay$
that preserves the timelike lines
$\cH$
and the time
$t$, $t(g(e),g(e'))=t(e,e')$
whenever
$t(e,e')$
is defined (we do not require to preserve the causality structure because this is
automatic).

\rem\label{rem:terminology} I am not satisfied with a terminology from
e.g. \cite{Bus}, \cite{PY}, where the term ``timelike (metric) space''
is used, because a respective object is never a metric space and its basic
feature is a causality relation. On the other hand to say ``timelike causal space'' 
sounds a little bit tautologically. Thus I use the term ``timed causal space'' instead.

\rem\label{rem:axioms_t5_t6} Strange looking Axioms~(t5), (t6) are automatically
satisfied for timed causal spaces induced by monotone M\"obius structures
on the circle, see sect.~\ref{sect:monotone}. However, they value and importance
are justified by the fact that (t5), (t6) are indispensable while one recovers 
a M\"obius structure on the circle from a timed causal space, see 
sect.~\ref{sect:timed_causal_moeb_structures}, especially
Lemma~\ref{lem:monb} and Lemma~\ref{lem:pro:timed_monotone:harmonic}.
In fact, Axiom~(t6) follows from the other axioms, see Appendix~2.

\section{M\"obius structures and hyperbolic spaces}
\label{sect:moeb}

On the boundary at infinity of any boundary continuous
Gromov hyperbolic space there is an induced  M\"obius 
structure. In this section, we recall details of this fact.

\subsection{Semi-metrics and topology}
\label{subsect:semi-metrics_topology}

Let
$X$
be a set. A function
$d:X^2\to\wh\R=\R\cup\{\infty\}$
is called a {\em semi-metric}, if it is symmetric,
$d(x,y)=d(y,x)$
for each
$x$, $y\in X$,
positive outside the diagonal, vanishes on the diagonal
and there is at most one infinitely remote point
$\om\in X$
for
$d$,
i.e. such that
$d(x,\om)=\infty$
for some
$x\in X\sm\{\om\}$.
Moreover, if
$\om\in X$
is such a point, then
$d(x,\om)=\infty$
for all 
$x\in X$, $x\neq\om$.
A metric is a semi-metric that satisfies the triangle inequality.

A 4-tuple
$q=(x_1,x_2,x_3,x_4)\in X^4$
is said to be {\em nondegenerate} if all its entries are pairwise
distinct. We denote by
$\reg\cP_4=\reg\cP_4(X)$
the set of ordered nondegenerate 4-tuples.

A {\em M\"obius structure}
$M$
on
$X$
is a class of M\"obius equivalent semi-metrics on
$X$,
where two semi-metrics are equivalent if and only if they have
the same cross-ratios on every
$q\in\reg\cP_4$.
An
$M$-{\em automorphism}
is a bijection
$f:X\to X$
that preserves cross-ratios. 

Given
$\om\in X$,
there is a semi-metric 
$d_\om\in M$
with infinitely remote point
$\om$.
It can be obtained from any semi-metric
$d\in M$
for which 
$\om$
is not infinitely remote by a {\em metric inversion},
$$d_\om(x,y)=\frac{d(x,y)}{d(x,\om)d(y,\om)}.$$
Such a semi-metric is unique up to a homothety, see \cite{FS},
and we use notation
$|xy|_\om=d_\om(x,y)$
for the distance between
$x$, $y\in X$
in that semi-metric. We also use notation
$X_\om=X\sm\{\om\}$.

Every M\"obius structure
$M$
on
$X$
determines the 
$M$-{\em topology}
whose subbase is given by all open balls centered at finite points
of all semi-metrics from
$M$
having infinitely remote points.

For the following fact see \cite[Corollary~4.3]{Bu1} in a more
general context of sub-M\"obius structures. We give here its proof for
convenience of the reader.

\begin{lem}\label{lem:continuity_semimetrics} For every
$\om\in X$, 
for a semi-metric
$d\in M$
with infinitely remote point 
$\om\in X$
and for every 
$x\in X_\om$
the function
$f_x:X\to\wh\R$, $f_x(y)=d(x,y)$,
is continuous in the 
$M$-topology.
\end{lem}

\begin{proof} The function
$f_x$
takes values in
$[0,\infty]$.
For 
$s,t\in[0,\infty]$
let
$B_s(x)=\set{y\in X}{$f_x(y)<s$}$
be the open
$d$-ball 
of radius
$s$
centered at
$x$, $C_t(x)=\set{y\in X}{$f_x(y)>t$}$
the complement of the closed
$d$-ball.
The inverse image
$f_x^{-1}(I)$
of any open interval
$I\sub[0,\infty]$
is either an open ball
$B_s(x)$,
or a complement
$C_t(x)$,
or an intersection
$B_s(x)\cap C_t(x)$
for some
$s>t$.

Let
$d_x\in M$
be the metric inversion of
$d$.
Then
$d_x(y,\om)=1/d(x,y)$,
hence
$C_t(x)=\set{y\in X}{$d_x(y,\om)<1/t$}$ 
is the open 
$d_x$-ball of radius
$1/t$
centered at
$\om$.
It follows that 
$f_x^{-1}(I)$
is open in the 
$M$-topology.
\end{proof}

\subsection{Boundary continuous hyperbolic spaces}
\label{subsect:boundary_continuous}

Let 
$Y$
be a metric space. Recall that the Gromov product
$(x|y)_o$
of
$x$, $y\in Y$
with respect to
$o\in Y$
is defined by
$$(x|y)_o=\frac{1}{2}(|xo|+|yo|-|xy|),$$
where 
$|xy|$
is the distance in
$Y$
between
$x$, $y$.
We use the following definition of a hyperbolic space
adapted to the case of geodesic metric spaces.

\begin{definition}\label{def:gromov_hyperbolic}
A geodesic metric space
$Y$
is {\em Gromov hyperbolic}, if for some
$\de\ge 0$
and any triangle
$xyz\sub Y$
the following holds: If
$y'\in xy$, $z'\in xz$
are points with 
$|xy'|=|xz'|\le(y|z)_x$,
then 
$|y'z'|\le\de$.
In this case, we also say that
$Y$
is
$\de$-hyperbolic,
and 
$\de$
is a {\em hyperbolicity constant} of
$Y$.
\end{definition}

A Gromov hyperbolic space
$Y$
is {\em boundary continuous} if the Gromov product extends
continuously onto the boundary at infinity 
$\di Y=X$
in the following
way: given
$\xi$, $\eta\in X$,
for any sequences
$\{x_i\}\in\xi$, $\{y_i\}\in\eta$
there is a limit
$(\xi|\eta)_o=\lim_i(x_i|y_i)_o$
for every
$o\in Y$,
for more details, see \cite[sect.~3.4.2]{BS1}. Note that in this case
$(\xi|\eta)_o$
is independent of the choice
$\{x_i\}\in\xi$, $\{y_i\}\in\eta$.
This allows one to define for every
$o\in Y$
a function
$(\xi,\eta)\mapsto d_o(\xi,\eta)=e^{-(\xi|\eta)_o}$,
which is a semi-metric on
$X$.

\begin{lem}\label{lem:semi-metrics_moeb} Let
$Y$
be a boundary continuous hyperbolic space. Then for any 
$o$, $o'\in Y$,
the semi-metrics
$d_o$, $d_{o'}$
on 
$X=\di Y$
are M\"obius equivalent.
\end{lem}

\begin{proof} Given 4-tuple
$(x,y,z,u)\sub Y$,
we put
$$\cd_o(x,y,z,u)=(x|u)_o+(y|z)_o-(x|z)_o-(y|u)_o$$
for a fixed
$o\in Y$.
Then
$\cd_o(x,y,z,u)=\cd(x,y,z,u)$
is independent of the choice of
$o$
because all entries containing
$o$
enter
$\cd_o(x,y,z,u)$
twice with the opposite signs. 

Now, given a nondegenerate 4-tuple
$q=(\al,\be,\de,\ga)\in\reg\cP_4(X)$,
for any 
$\{x_i\}\in\al$, $\{y_i\}\in\be$, $\{z_i\}\in\ga$, $\{u_i\}\in\de$,
by the boundary continuity of
$Y$,
there is a limit
$$\cd(\al,\be,\ga,\de)=\lim_i\cd(x_i,y_i,z_i,u_i),$$
which coincides with
$(\al|\de)_o+(\be|\ga)_o-(\al|\ga_o)-(\be|\de)_o$.
Thus the cross-ratio
$$\frac{d_o(\al,\ga)d_o(\be,\de)}{d_o(\al,\de)d_o(\be,\ga)}=\exp(-\cd(\al,\be,\ga,\de))$$
is independent of 
$o$.
Hence, semi-metrics
$d_o$, $d_{o'}$
are M\"obius equivalent for any 
$o$, $o'\in Y$.
\end{proof}

The M\"obius structure
$M$
on the boundary at infinity
$X=\di Y$
of a boundary continuous hyperbolic space
$Y$
generated by any semi-metric
$d_o(\xi,\eta)=\exp(-(\xi|\eta)_o)$, $o\in Y$,
is said to be {\em induced} (from
$Y$).
 For any 
$\om\in X$, $o\in Y$,
the metric inversion
$d_\om$
of
$d_o$
with respect to
$\om$
is a semi-metric on
$X$
from
$M$
with the infinitely remote point
$\om$.
Recall that any two semi-metrics in
$M$
with a common infinitely remote point are proportional to each other.
Thus metric inversions with respect to
$\om$ 
of semi-metrics
$d_o$, $d_{o'}$
are proportional to each other for any 
$o$, $o'\in Y$.

In {\bf Appendix}, we show that every proper Gromov hyperbolic
$\CAT(0)$
space is boundary continuous, see Theorem~\ref{thm:cat0_boundary_contiuous}.

\section{Monotone M\"obius structures on the circle}
\label{sect:monotone}

\subsection{Axioms for monotone M\"obius structures on the circle}
\label{subsect:axioms_monotone}

We say that a M\"obius structure
$M$
on
$X=S^1$
is {\em monotone}, if it satisfies the following Axioms
\begin{itemize}
 \item [(T)] Topology: $M$-topology
on
$X$
is that of
$S^1$; 
\item[(M)] Monotonicity: given a 4-tuple
$q=(x,y,z,u)\in X^4$
such that the pairs
$(x,y)$, $(z,u)$
separate each other,
we have
$$|xy|\cdot|zu|>\max\{|xz|\cdot|yu|,|xu|\cdot|yz|\}$$
for some and hence any semi-metric from
$M$.
\end{itemize}

\begin{rem}\label{rem:ref_schroeder} These Axioms have arisen in
a discussion with V.~Schroeder while working on \cite{BS4}. 
\end{rem}

A choice of
$\om\in X$
uniquely determines the interval
$xy\sub X_\om$
for any distinct
$x$, $y\in X$
different from
$\om$
as the arc in
$X$
with the end points
$x$, $y$
that does not contain
$\om$.
As an useful reformulation of Axiom~(M) we have

\begin{cor}\label{cor:interval_monotone} Assume for a nondegenerate
4-tuple
$q=(x,y,z,u)\in\reg\cP_4$
the interval
$xz\sub X_u$
is contained in
$xy$, $xz\sub xy\sub X_u$.
Then
$|xz|_u<|xy|_u$.
\end{cor}

\begin{proof} By the assumption, the pairs
$(x,y)$, $(z,u)$
separate each other. Hence, by Axiom~(M) we have
$|xz||yu|<|xy||zu|$
for any semi-metric from
$M$.
In particular,
$|xz|_u<|xy|_u$.
\end{proof}

We denote by
$\cM$
the class of monotone M\"obius structures on 
$S^1$.

\subsection{Examples of monotone M\"obius structures on the circle}
\label{subsect:examples_monotone}

By Theorem~\ref{thm:cat0_boundary_contiuous}, every proper
Gromov hyperbolic
$\CAT(0)$
space 
$Y$
is boundary continuous, and thus 
$\di Y$
possesses an induced M\"obius
structure.

Recall that in any 
$\CAT(0)$
space
$Y$,
the angle 
$\angle_o(x,x')$
between geodesic segments 
$ox$, $ox'$
with a common vertex
$o$
is well defined and by definition it is at most
$\pi$, $\angle_o(x,x')\le\pi$.

A point
$o$
in a 
$\CAT(0)$
space
$Y$
with
$\di Y$
homeomorphic to the circle
$S^1$
is said to be {\em singular,}
if there are two geodesics
$\xi\xi'$, $\eta\eta'\sub Y$
through
$o$
such that the pairs of points
$(\xi,\xi')$
and 
$(\eta,\eta')$
in
$\di Y$
separate each other, and 
$\angle_o(\xi,\eta)+\angle_o(\xi',\eta')\ge 2\pi$.

\begin{thm}\label{thm:without_singular} Let
$Y$
be a Gromov hyperbolic 
$\CAT(0)$
surface with
$\di Y=S^1$
and without singular points. Then the induced M\"obius structure
$M$
on
$X=\di Y$
is monotone.
\end{thm}

\begin{proof} For the induced M\"obius structure, the 
$M$-topology
on
$X$
coincides with the standard Gromov topology, see \cite[Sect.~2.2.3]{BS1} or
\cite[Lemma~5.1]{Bu1}.
Thus
$M$
satisfies Axiom~(T) by the assumption. To check Axiom~(M), consider a 4-tuple
$q=(\xi,\xi',\eta,\eta')\in X^4$
such that the pairs
$(\xi,\xi')$
and 
$(\eta,\eta')$
separate each other. Since
$Y$
is Gromov hyperbolic, there are geodesics
$\xi\xi'$, $\eta\eta'\sub Y$
with the end points at infinity
$\xi,\xi'$
and 
$\eta,\eta'$
respectively. The assumption on separation and the fact that
$Y$
is a 
$\CAT(0)$
surface imply that these geodesics intersect at some point
$o$.
We have
$(\xi|\xi')_o=0=(\eta|\eta')_o$.
Thus
$|\xi\xi'|=1=|\eta\eta'|$
for the semi-metric
$|xy|=\exp(-(x|y)_o)$
on
$X$.
Recall that this semi-metric is a semi-metric of 
$M$.

The angles at
$o$
between the rays
$ox$, $x=\xi,\eta,\xi',\eta'$
in this cyclic order, form two opposite pairs.
Since 
$o$
is not singular in
$Y$,
at least one of angles
$\angle_o(x,z)<\pi$
for each opposite pair.
By Corollary~\ref{cor:zero_gromov_product},
$(x|z)_o>0$,
thus
$|xz|<1$.
It follows that 
$$|\xi\xi'|\cdot|\eta\eta|>\max\{|\xi\eta|\cdot|\xi'\eta'|,
                                 |\xi'\eta|\cdot|\xi\eta'|\},$$
i.e.,
$M$
satisfies Axiom~(M).
\end{proof}

\begin{exas}\label{exas:monotone_nonmonotone} 1. The canonical M\"obius
structure 
$M_0$
on the circle is monotone.

2. Let
$S$
be a closed surface of negative Euler characteristic with an Euclidean metric 
having cone type singularities with
complete angles
$>2\pi$
about every singular point,
$Y$
the universal covering of 
$S$
with the lifted metric. Then
$Y$
is a Gromov hyperbolic 
$\CAT(0)$
surface with 
$\di Y=S^1$.
It follows from Theorem~\ref{thm:cat0_boundary_contiuous} that
$Y$
induces a M\"obius structure
$M$
on
$\di Y$.
However, 
$M$ 
is not monotone.

3. Absence of singular points on a Gromov hyperbolic
$\CAT(0)$
surface 
$Y$
does not mean that
$Y$
has no metric singularities. Remove an open equidistant neighborhood
of a geodesic line in
$\hyp^2$
and glue remaining pieces by an isometry between their boundaries.
Then the obtained
$Y$
is a proper Gromov hyperbolic
$\CAT(-1)$
surface with
$\di Y=S^1$
without singular points, and Theorem~\ref{thm:without_singular}
can be applied to
$Y$.
At the same time, 
$Y$
has metric singularities along the gluing line.

4. Every (topological) embedding
$f:S^1\to\wh\R^2$
induces on
$S^1$
some metric and therefore a M\"obius structure
$M_f$.
For 
$f=\id$,
the M\"obius structure
$M_f$
coincides with the canonical one, and thus it is monotone. However, if
$f(S^1)\sub\R^2$
is an ellipse with principal semi-axes
$a$, $b$
such that 
$4ab\le a^2+b^2$,
then
$M_f$
is not monotone. That is, the M\"obius structure induced on a convex
curve in
$\wh\R^2$
in general is not monotone. 
\end{exas}

In what follows, we assume that a monotone M\"obius structure
$M$
on
$X=S^1$
is fixed.

\subsection{Harmonic pairs}
\label{subsect:harmonic_events}

A pair
$a=(x,y)$, $b=(z,u)\in\ay$
of events is 
{\em $M$-harmonic} or form an 
$M$-harmonic 4-tuple, if 
\begin{equation}\label{eq:harmonic_events}
|xz|\cdot|yu|=|xu|\cdot|yz|
\end{equation}
for some and hence any semi-metric of the M\"obius structure
(this is the same as (\ref{eq:harmonic}). However, for a general M\"obius
structure
$M$
on
$S^1$,
the interpretation of (\ref{eq:harmonic}) for the canonical
$M_0$
as orthogonality of respective geodesic lines in
$\hyp^2$
makes no sense). Nevertheless, in this case we also say that 
$a$, $b$
are mutually {\em orthogonal,}
$a\perp b$.
Note that any harmonic 4-tuple
$q=(a,b)$
is nondegenerate.

\begin{lem}\label{lem:harm_separate} If events
$a=(x,y)$, $b=(z,u)\in\ay$
are mutually orthogonal,
$a\perp b$,
then they separate each other and
$z\in xy\sub X_u$
is a unique midpoint.
\end{lem}

\begin{proof} We have
$|xz|_u=|zy|_u$.
By Corollary~\ref{cor:interval_monotone},
$z\in xy\sub X_u$
is a unique midpoint, and thus
$a$, $b$
separate each other.
\end{proof}

\begin{lem}\label{lem:harm_exist} For every
$e\in\ay$
and every
$x\in X\sm e$
there is a unique
$y\in X$
such that the pair 
$a=(x,y)$, $e\in\ay$
of events is harmonic.
\end{lem}

\begin{proof} Let
$e=(z,u)$.
By Axiom~(T) and Lemma~\ref{lem:continuity_semimetrics}, the functions
$f_z,f_u:X\to\wh\R=\R\cup\{\infty\}$,
$f_z(t)=|zt|_x$, $f_u(t)=|ut|_x$
are continuous on
$X=S^1$,
and they take values between
$0=f_z(z)=f_u(u)$
and 
$f_z(u)=f_u(z)>0$
on the segment 
$zu\sub X_x$.
Thus there is a midpoint
$y\in zu\sub X_x$
between
$z$
and
$u$, $|zy|_x=|yu|_x$.
By Corollary~\ref{cor:interval_monotone}, such a point
$y$
is unique.
\end{proof}

\begin{pro}\label{pro:monotone_moeb_to_timed_space} For every
monotone M\"obius structure
$M\in\cM$
there is a uniquely determined timed causal space
$T=\wh T(M)\in\cT$
such that the automorphism group
$G_M$
of
$M$
injects into the automorphism group 
$G_T$
of
$T$:
If
$g:X\to X$
is an
$M$-M\"obius 
automorphism, then
the induced
$\wh g:\ay\to\ay$
is an automorphism of
$T$. 
\end{pro}

We prove Proposition~\ref{pro:monotone_moeb_to_timed_space} 
in sections~\ref{subsect:timelike_lines} and \ref{subsect:time_events}.

\subsection{Timelike lines}
\label{subsect:timelike_lines}

Any timelike line
$h$
in
$\ay$
is associated with an event
$e\in\ay$, $h=h_e$,
and is defined as the set of events
$a\in\ay$
such that the pair
$(a,e)$
is harmonic,
$h_e=\set{a\in\ay}{$a\perp e$}$. 
We denote by
$\cH=\cH_M$
the collection of timelike lines in
$\ay$.

\begin{pro}\label{pro:safisfies_hypline_axiom} The collection
$\cH$
satisfies Axioms~(h1)--(h5).
\end{pro}

\begin{proof} Axioms~(h1), (h5) hold by definition, (h2)
follows from Lemma~\ref{lem:harm_separate}, (h4) from 
Lemma~\ref{lem:harm_exist}.

To check (h3), assume
$e=(z,u)\in\ay$, $a=(x,y)$, $a'=(x',y')\in h_e$.
Then
$z$
is the midpoint of the segments
$xy$, $x'y'\sub X_u$.
Thus by Axiom~(M), the pairs
$(x,y)$, $(x',y')$
do not separate each other, that is, the events
$a$
and
$a'$
are in the causal relation. Hence, (h3).
\end{proof}

We say that an event
$a\in\ay$
is a {\em common perpendicular} to events
$e$, $e'\in\ay$,
if
$e$, $e'\in h_a$.

\begin{cor}\label{cor:common_perpendicular_prescribed} Given
$e$, $e'\in\ay$
in the strong causal relation, there is a common perpendicular
$a\in\ay$
to
$e$, $e'$.
\end{cor}

\begin{proof} By Proposition~\ref{pro:safisfies_hypline_axiom},
the collection
$\cH=\cH_M$
of timelike lines in
$\ay$
determined by the M\"obius structure
$M$
satisfies Axioms~(h1)--(h5). Thus the assertion follows from
Proposition~\ref{pro:timelike_lines}(b), see Remark~\ref{rem:axiom_h6}.
\end{proof}

The proof of (h6) we postpone to sect.~\ref{subsect:time_events},
see Lemma~\ref{lem:unique_common_perpendicular}.  

\subsection{Time between events}
\label{subsect:time_events}

The time between events
$a$, $a'\in\ay$
is defined if and only they are in the causal relation. We do this
essentially as in sect.~\ref{subsect:recoveting_desitter} using
formula (\ref{eq:distance_log}). First of all, the time between
events on a light line by definition is zero,
$t(a,a')=0$
for 
$a$, $a'\in p_x$, $x\in X$.

Next, assume that
$a$, $a'\in\ay$
are in the strong causal relation. Then by Corollary~\ref{cor:common_perpendicular_prescribed}
there is a common perpendicular 
$e\in\ay$
to
$a$, $a'$,
that is,
$a$, $a'\in h_e$.
We let
$e=(x,y)$, $a=(z,u)$, $a'=(z',u')$.
Then by definition
\begin{equation}\label{eq:time}
t_e(a,a')=\left|\ln\frac{|xz'|\cdot|yz|}{|xz|\cdot|yz'|}\right|
\end{equation}
for some and hence any semi-metric on
$X$
from
$M$.
It follows from harmonicity of
$(a,e)$
and
$(a',e)$
that 
\begin{equation}\label{eq:time_different}
t_e(a,a')=\left|\ln\frac{|xu'|\cdot|yu|}{|xu|\cdot|yu'|}\right|=
             \left|\ln\frac{|xu'|\cdot|yz|}{|xz|\cdot|yu'|}\right|=
             \left|\ln\frac{|xz'|\cdot|yu|}{|xu|\cdot|yz'|}\right|,
\end{equation}
and we often use these different representations of
$t_e(a,a')$.

\begin{lem}\label{lem:unique_common_perpendicular} Given distinct
$a$, $a'\in\ay$
in the causal relation, there is at most one common perpendicular
$b\in\ay$
to
$a$, $a'$.
In particular, the time
$t(a,a')=t_e(a,a')$
is well defined, and Axiom~(h6) is fulfilled for the collection
$\cH$
of timelike lines.
\end{lem}

\begin{proof} The idea is taken from \cite{BS4}. If events
$a$, $a'$
are on a light line, then they cannot lie on a timelike line,
say 
$h_e$,
because by Axiom~(h1) they both must separate
$e$,
which would contradict (h4). Thus we assume that
$a$, $a'$
are not on a light line.

Assume there are common perpendiculars
$b=(z,u)$, $b'=(z',u')\in\ay$
to
$a$, $a'$, 
that is
$b\perp a,a'$ 
and
$b'\perp a,a'$,
or which is the same,
$b$, $b'\in h_a\cap h_{a'}$.
By already established Axiom~(h3), see Proposition~\ref{pro:safisfies_hypline_axiom},
$b$
and 
$b'$
do not separate each other. 
Let
$a=(x,y)$, $a'=(x',y')$.
Without loss of generality, we assume that on 
$X_x$
we have the following order of points
$zz'yy'u'ux'$.

By Axiom~(h5),
$a$, $a'\in h_b\cap h_{b'}$.
The times
$t=t_b(a,a')$, $t'=t_{b'}(a,a')$
are already defined by (\ref{eq:time}). Computing them in 
a semi-metric of the M\"obius structure with infinitely remote point
$x$,
we obtain
$$e^t=\frac{|zx'|}{|x'u|},\quad
  e^{t'}=\frac{|z'x'|}{|x'u'|}.$$
Using the order of points
$zz'yy'u'ux'$
on
$X_x$,
we have, in particular, that the interval
$z'x'$
is contained in the interval
$zx'$.
By Corollary~\ref{cor:interval_monotone},
$|zx'|\ge|z'x'|$.
Similarly,
$x'u\sub x'u'$
and hence
$|x'u|\le|x'u'|$.
Thus
$t\ge t'$
and if 
$b'\neq b$,
the inequality is strong. Applying this argument with infinitely remote point
$y$,
we obtain
$t\le t'$.
Therefore
$t=t'$
and
$b=b'$. 
\end{proof}

\begin{pro}\label{pro:satisfies_time_axioms} The time
between events in
$\ay$
defined above satisfies Axioms~(t1)--(t6). 
\end{pro}

\begin{proof} Axiom~(t1) is satisfied by the definition of the time
$t$.

Axiom~(t2): If events
$a$, $a'$
are on a light line, then
$t(a,a')=0$
by definition. Conversely, assume
$t(a,a')=0$
for events
$a$, $a'\in\ay$
in the causal relation, which are not on a light line, in particular,
$a\neq a'$.
Then by Lemmas~\ref{cor:common_perpendicular_prescribed}
and \ref{lem:unique_common_perpendicular}, there is a unique
$e\in\ay$
with
$a$, $a'\in h_e$. 
We let
$e=(x,y)$, $a=(z,u)$, $a'=(z',u')$.
Since
$a\neq a'$,
we have
$z'\neq z,u$.
On the other hand, it follows from (\ref{eq:time}) that
$|xz'|\cdot|yz|=|xz|\cdot|yz'|$
for any semi-metric on
$X$
from
$M$.
In particular,
$|xz'|_y=|xz|_y$,
and by monotonicity (M),
$x$
is the midpoint between
$z$, $z'$
in
$X_y$.
Hence,
$\rho_e(z)=z'=u$,
a contradiction.

Axiom~(t3) follows from the definition of the time
$t$
and (\ref{eq:time}).

Axiom~(t4a): Let
$e$, $e'$, $e''\in h_a$
with 
$e\le e'\le e''$.
If
$e'$
coincides with 
$e$
or
$e''$,
then the required equality is trivial. Thus we assume that
$e<e'<e''$.
Without loss of generality, we can assume that for 
$a=(x,y)$, $e=(z,u)$, $e'=(z',u')$, $e''=(z'',u'')$,
the points
$z$, $z'$, $z''$
lie on one and the same arc determined by
$a$
in the order
$xzz'z''y$.
Then
$$\exp(t(e,e'))=\frac{|xz'|\cdot|yz|}{|xz|\cdot|yz'|},\quad 
  \exp(t(e',e''))=\frac{|xz''|\cdot|yz'|}{|xz'|\cdot|yz''|},$$
and we obtain
$$\exp(t(e,e')+t(e',e''))=\frac{|xz''|\cdot|yz|}{|xz|\cdot|yz''|}
  =\exp(t(e,e'')).$$

Axiom~(t4b): Given an event
$e=(z,u)$
on a timelike line
$h_a\sub\ay$
with 
$a=(x,y)$,
and 
$s>0$,
we take a semi-metric from
$M$
with the infinitely remote point
$y$.
Then
$|xz|_y=|xu|_y:=t>0$,
and without loss of generality, we can assume that 
$t=1$.
Note that the function
$f_x:X_y\to\R$, $f_x(x')=|xx'|_y$,
is continuous, see Lemma~\ref{lem:continuity_semimetrics},
and monotone, see Corollary~\ref{cor:interval_monotone}.
It varies from
$0=f_x(x)$
to
$\infty=f_x(y)$.
Thus there are 
$z_-\in xz$, $z_+\in zy$
with 
$f_x(z_\pm)=e^{\pm s}$.
For the events
$e_\pm=(z_\pm,u_\pm)\in h_a$,
where 
$u_\pm=\rho_a(z_\pm)$,
we have 
$$t(e,e_\pm)=\left|\ln\frac{|xz_\pm|_y}{|xz|_y}\right|=|\ln|xz_\pm|_y|=s.$$
Choosing a decomposition
$X=e^+\cup e^-$
determined by 
$e$
so that 
$y\in e^+$,
we have
$e_\pm\in h_a\cap C_e^\pm$
and 
$t(e,e_\pm)=s$.

Axiom~(t5): Let
$e=(x,y)$, $d=(z,u)$
be events in the strong causal relation. Then we have by (\ref{eq:time}), 
(\ref{eq:time_different})
$$t(z_e,u_e)=\left|\ln\frac{|xu|\cdot|yz|}{|xz|\cdot|yu|}\right|=t(x_d,y_d).$$

Axiom~(t6): Given
$e=(x,y)\in\ay$, $d=(z,u)\in h_e$, 
we put
$a=(x,z)$, $b=(x,u)$.
Then we have by (\ref{eq:time}), (\ref{eq:time_different})
$$t(y_a,u_a)=\left|\ln\frac{|xy|\cdot|zu|}{|xu|\cdot|zy|}\right|,\quad 
  t(y_b,z_b)=\left|\ln\frac{|xy|\cdot|zu|}{|xz|\cdot|yu|}\right|.$$
On the other hand, the pair
$(a,e)$
is harmonic, thus
$|xu|\cdot|zy|=|xz|\cdot|yu|$.
Hence,
$t(y_a,u_a)=t(y_b,z_b)$.
\end{proof}

\begin{proof}[Proof of Proposition~\ref{pro:monotone_moeb_to_timed_space}]
Given a monotone M\"obius structure
$M\in\cM$,
we have defined above a class
$\cH=\cH_M$
of timelike lines in
$\ay$
that satisfies Axioms~(h1)--(h6), and a time
$t$
on
$(\ay,\cH)$
that satisfies Axioms~(t1)--(t6). Therefore, a timed causal space
$T=(\ay,\cH,t)$, $T=\wh T(M)$,
is defined.

Let
$g:X\to X$
be an
$M$-M\"obius
automorphism. Then the induced
$\wh g:\ay\to\ay$
preserves the causality structure, the class of timelike lines
$\cH$
and the time 
$t$
because the last two are defined via cross-ratios. Thus
$\wh g$
is an automorphism of
$T$.
If
$\wh g=\id_T$,
then, in particular, it preserves every light line.
Hence,
$g=\id$,
and the group
$G_M$
of the
$M$-automorphisms
injects into the group 
$G_T$
of the
$T$-automorphisms. 
\end{proof}

\section{Timed causal spaces and M\"obius structures}
\label{sect:timed_causal_moeb_structures}

In this section we adopt the following more advanced point of view to
M\"obius structures, see \cite{Bu1}.

\subsection{M\"obius and sub-M\"obius structures}
\label{subsect:moeb_sub-moeb}

Let
$X$
be a set,
$\reg\cP_4=\reg\cP_4(X)$,
see sect.~\ref{subsect:semi-metrics_topology}. For any semi-metric
$d$
on
$X$
we have three cross-ratios
$$q\mapsto \crr_1(q)=\frac{|x_1x_3||x_2x_4|}{|x_1x_4||x_2x_3|};
  \crr_2(q)=\frac{|x_1x_4||x_2x_3|}{|x_1x_2||x_3x_4|};
  \crr_3(q)=\frac{|x_1x_2||x_3x_4|}{|x_2x_4||x_1x_3|}$$
for 
$q=(x_1,x_2,x_3,x_4)\in\reg\cP_4$,
whose product equals 1, where
$|x_ix_j|=d(x_i,x_j)$.
We associate with 
$d$
a map 
$M_d:\reg\cP_4\to L_4$
defined by
\begin{equation}\label{eq:moeb_map}
M_d(q)=(\ln\crr_1(q),\ln\crr_2(q),\ln\crr_3(q)),
\end{equation}
where
$L_4\sub\R^3$
is the 2-plane given by the equation
$a+b+c=0$.

Two semi-metrics
$d$, $d'$
on
$X$
are M\"obius equivalent if and only
$M_d=M_{d'}$.
Thus a M\"obius structure on
$X$
is completely determined by a map 
$M=M_d$
for any semi-metric
$d$
of the M\"obius structure, and we often identify a M\"obius structure
with the respective map 
$M$.
A bijection
$f:X\to X$
is an 
$M$-automorphism
if and only if
$M\circ\ov f(q)=M(q)$
for every ordered 4-tuple
$q\in\reg\cP_4$
for the induced
$\ov f:\reg\cP_4\to\reg\cP_4$.

Let 
$S_n$
be the symmetry group of 
$n$
elements. The group
$S_4$
acts on
$\reg\cP_4$
by entries permutations of any
$q\in\reg\cP_4$.
The group
$S_3$
acts on
$L_4$
by signed permutations of coordinates, where a permutation
$\si:L_4\to L_4$
has the sign 
``$-1$'' 
if and only if
$\si$
is odd. 

The {\em cross-ratio} homomorphism
$\phi:S_4\to S_3$
can be described as follows: a permutation of a tetrahedron
ordered vertices 
$(1,2,3,4)$
gives rise to a permutation of pairs of opposite edges
$((12)(34),(13)(24),(14)(23))$. 
Thus the kernel 
$K$
of
$\phi$
consists of four elements 1234, 2143, 4321, 3412,
and is isomorphic to the dihedral group
$D_4$
of a square automorphisms. We denote by 
$\sign:S_4\to\{\pm 1\}$
the homomorphism that associates to every odd permutation the sign 
``$-1$''.

One easily check that any M\"obius structure 
$M:\reg\cP_4\to L_4$
is equivariant with respect to the signed cross-ratio homomorphism,
\begin{equation}\label{eq:signed_cross-ratio_homomorphism}
M(\pi(q))=\sign(\pi)\phi(\pi)M(q) 
\end{equation}
for every
$q\in\reg\cP_4$, $\pi\in S_4$,
where
$\phi:S_4\to S_3$
is the cross-ratio homomorphism.

A {\em sub-M\"obius structure} on
$X$
is a map 
$M:\cP_4\to L_4$
with the basic property (\ref{eq:signed_cross-ratio_homomorphism})
(we drop details related to degenerated 4-tuples, which can be found in
\cite{Bu1}).
Now, we describe a criterion for a sub-M\"obius structure to be
M\"obius. Given an ordered tuple
$q=(x_1,\dots,x_k)\in X^k$,
we use notation
$$q_i=(x_1,\dots,x_{i-1},x_{i+1},\dots,x_k)$$
for
$i=1,\dots,k$.

For a sub-M\"obius structure
$M$
on
$X$
we define its {\em codifferential}
$\de M:\reg\cP_5\to L_5=L_4^5$
by
$$(\de M(q))_i=M(q_i),\ i=1,\dots,5.$$
Furthermore, we use notation
$M(q_i)=(a(q_i),b(q_i),c(q_i))$, $i=1,\dots,5$, $q\in\reg\cP_5$.
The following theorem has been proved in \cite[Theorem~3.4]{Bu1}.

\begin{thm}\label{thm:submoeb_moeb} A sub-M\"obius structure
$M$
on
$X$
is a M\"obius structure if and only if for every nondegenerate 5-tuple
$q\in X^5$
the following conditions (A), (B) are satisfied
\begin{itemize}
 \item [(A)] $b(q_1)+b(q_4)=b(q_3)-a(q_1)$;
 \item[(B)] $b(q_2)=-a(q_4)+b(q_1)$.
\end{itemize}
\end{thm}

\begin{rem}\label{rem:equivalent_conditions} Conditions (A) and (B)
are in fact equivalent to each other. This follows from 
$S_5$-symmetry
of the codifferential
$\de M$
and is explained in details
in \cite{Bu2}.
\end{rem}

\subsection{Timed causal space and a sub-M\"obius structure}
\label{subsect:timed_sub-moeb}

We begin with
\begin{rem}\label{rem:equivalent_monotone}
 Axiom~(M) for monotone M\"obius structures on the circle is equivalent to that
$a(q)<0$
and
$b(q)>0$,
where
$M(q)=(a(q),b(q),c(q))\in L_4$
for any 
$q=(x,y,z,u)\in\reg\cP_4$
such that the pairs
$(x,u)$
and 
$(y,z)$
separate each other. This follows from Eq.~\ref{eq:moeb_map}.
\end{rem}

With every timed causal space
$T=(\ay,\cH,t)$
we associate a sub-M\"obius structure
$M$
on
$X=S^1$
as follows. 

We fix an orientation of
$S^1$.
Then for any 4-tuple
$q\in\reg\cP_4$
we have a well defined {\em cyclic order}
$\co(q)$.
Let
$A\sub\cP_4$
be the set 
$\set{\pi q}{$\pi\in S_4$}$, $A=A(q)$,
for a given
$q\in\cP_4$.
Note that the cyclic order
$\co(\pi q)=\co(q)$
is independent of
$\pi\in S_4$,
and we denote it by
$\co(A)$.

We label
$\co(A)=1234$.
With any pair
$i,i+1$
for consecutive points in
$\co(A)$,
we associate the timelike line
$h_{i,i+1}$
dual to the event
$(i,i+1)$.
Two other points of
$\co(A)$
determine the events
$a=(i+2)_{(i,i+1)}$, $a'=(i+3)_{(i,i+1)}\in h_{i,i+1}$,
and we associate with 
$(i,i+1)$
the time
$t_{i(i+1)}>0$
between
$a$, $a'$.
In that way, every pair 
$(i,i+1)$
of consecutive points in
$\co(A)$
is labeled by a positive time 
$t_{i(i+1)}$. 

Adjacent pairs are labeled in general by distinct numbers
$t_{i(i+1)}$, $t_{(i+1)(i+2)}$, however, by Axiom~(t5), the opposite pairs
are labeled by one and the same number, that is,
$t_{i(i+1)}=t_{(i+2)(i+3)}$,
where indexes are taken modulo 4.

Assume that a 4-tuple
$q=(x,y,z,u)\in\reg\cP_4$
is obtained from
$\co(A)$
by fixing the initial point and by the transposition of two last
entries of
$\co(A)$,
that is,
$x$
is the chosen initial point and
$\co(A)=xyuz$.
Then we put
\begin{equation}\label{eq:def_moeb_timed}
M(q)=(-t_{xy},t_{yu},t_{xy}-t_{yu})\in L_4. 
\end{equation}

We denote by
$B\sub A$
the subset consisting of all 4-tuples obtained from
$\co(A)$
by fixing an initial entry and transposing
two last entries of
$\co(A)$.
The set 
$B$
consists of 4 elements,
$|B|=4$,
and it is an orbit of a cyclic subgroup
$\Ga_\pi\sub S_4$,
generated by the permutation
$\pi=2413\in S_4$, 
that is,
$\Ga_\pi=\{\id,\pi,\pi^2,\pi^3\}$
and
$B=\set{\si q}{$\si\in\Gamma_\pi$}$
for any 
$q\in B$.
For example, if
$\co(A)=xyuz$,
then
$q=(x,y,z,u)$, $\pi q=(y,u,x,z)$,
$\pi^2q=(u,z,y,x)$, $\pi^3q=(z,x,u,y)\in B$.

\begin{lem}\label{lem:monb} For any 
$q$, $q'=\si q\in B$
with
$\si\in\Ga_\pi$
we have
$$M(q')=\sign(\si)\phi(\si)M(q).$$
\end{lem}

\begin{proof} It suffices to prove the equality for the generator
$\si=2413$
of
$\Ga_\pi$.
Assume without loss of generality that 
$q=(x,y,z,u)$
and, hence,
$\co(A)=xyuz$.
We also write
$\co(A)=1234$.
Then
$q'=\si q=(y,u,x,z)$.
By definition,
$M(q)=(-t_{12},t_{23},t_{12}-t_{23})$, $M(q')=(-t_{23},t_{34},t_{23}-t_{34})$.
On the other hand,
$\sign(\si)=-1$
because
$\si$
is odd, and
$\phi(\si)=213$.
Therefore,
$$\sign(\si)\phi(\si)M(q)=-(t_{23},-t_{12},t_{12}-t_{23})=(-t_{23},t_{12},t_{23}-t_{34})
=M(q'),$$
because
$t_{12}=t_{34}$
by Axiom~(t5).
\end{proof}

Furthermore, for every
$p\in A$
there are
$\si\in S_4$
and
$q\in B$
such that 
$p=\si q$.
We put
\begin{equation}\label{eq:equivar}
M(p)=\sign(\si)\phi(\si)M(q). 
\end{equation}

\begin{pro}\label{pro:submoebius_welldefined} The equation (\ref{eq:equivar})
defines unambiguously a map
$M:\reg\cP_4\to L_4$
which is a sub-M\"obius structure on
$X$.
\end{pro}

\begin{proof} We show that for a different representation
$p=\si'q'$
with
$\si'\in S_4$, $q'\in B$,
Eq.~(\ref{eq:equivar}) gives the same value
$M(p)$.
We have
$\si'q'=\si q$,
thus
$q'=\rho q$
with 
$\si'\rho=\si$.
Since
$q$, $q'\in B$
and the group
$S_4$
acts on
$A$
effectively, we have
$\rho\in\Gamma_\pi$.
Then by Lemma~\ref{lem:monb}
$$M(q')=\sign(\rho)\phi(\rho)M(q),$$
and we obtain 
$$\sign(\si')\phi(\si')M(q')=\sign(\si'\rho)\phi(\si'\rho)M(q)=M(p).$$
Thus Eq.~(\ref{eq:equivar}) defines unambiguously a map
$M:\reg\cP_4\to L_4$,
which now satisfies (\ref{eq:equivar}) for any 
$p=\si q$
with
$q\in\reg\cP_4$, $\si\in S_4$.
Hence,
$M$
is a sub-M\"obius structure on
$X$.
\end{proof}

Note that to define the sub-M\"obius structure
$M$
we do not use Axiom~(t6).

\subsection{The sub-M\"obius structure $M$ is a M\"obius one}
\label{subsect:sub-moeb_moeb}

Given a nondegenerate 5-tuple
$q\in\reg\cP_5$,
we label its cyclic order by
$\co(q)=12345$.
Assuming that the order of
$q=xyzuv$
is cyclic, we have
$\co(q_i)=\co(q)_i$
for 
$i=1,\dots,5$.
We consider with every
$i\in\co(q)$
three variables
$t_{(i+1)(i+2)}^i$, $t_{(i+2)(i+3)}^i$, $t_{(i+3)(i+4)}^i$,
associated with the 4-tuple
$\co(q_i)=\co(q)_i$
as in sect.~\ref{subsect:timed_sub-moeb}, where indexes are 
taken modulo 5. These 15 variables satisfy 10 equations
\begin{equation}\label{eq:time_(t4)}
t_{(i+1)(i+2)}^i=t_{(i+3)(i+4)}^i 
\end{equation}
\begin{equation}\label{eq:time_(t6)}
t_{(i+2)(i+3)}^i=t_{(i+2)(i+3)}^{i+1}+t_{(i+2))i+3)}^{i+4},
\end{equation}
which follow from Axioms~(t5) and (t4a) respectively. We compute
$\de M(q)=v$
for 
$q\in\reg\cP_5$
with
$\co(q)=12345$
as follows. The time-labeling of
$\co(q_i)=\co(q)_i$
is given by
$t_{(i+1)(i+2)}^i$, $t_{(i+2)(i+3)}^i$, $t_{(i+3)(i+4)}^i$, $t_{(i+4)(i+6)}^i$.
Thus according to our definition of the sub-M\"obius structure 
$M$
we have
$$(-t_{(i+1)(i+2)}^i,t_{(i+2)(i+3)}^i,t_{(i+1)(i+2)}^i-t_{(i+2)(i+3)}^i)=M(\pi \si^{i-1}q_i),$$
where 
$\pi=1243$, $\si=4123$.
Therefore,
$$M(\si^{i-1}q_i)=\sign(\pi)\phi(\pi)M(\pi\si^{i-1}q_i)
        =(t_{(i+1)(i+2)}^i,t_{(i+2)(i+3)}^i-t_{(i+1)(i+2)}^i,-t_{(i+2)(i+3)}^i),$$
and we obtain
$$\de M(q)=\begin{pmatrix}
     a_1&b_1&c_1\\
     a_2&b_2&c_2\\
     a_3&b_3&c_3\\
     a_4&b_4&c_4\\
     a_5&b_5&c_5
    \end{pmatrix}
    =\begin{pmatrix}
    t_{23}^1 &t_{34}^1-t_{23}^1 &-t_{34}^1\\
    t_{45}^2 &t_{34}^2-t_{45}^2 &-t_{34}^2\\
    t_{12}^3 &t_{15}^3-t_{12}^3 &-t_{15}^3\\
    t_{12}^4 &t_{23}^4-t_{12}^4 &-t_{23}^4\\
    t_{12}^5 &t_{23}^5-t_{12}^5 &-t_{23}^5    
    \end{pmatrix}.$$

\begin{thm}\label{thm:submoeb_is_moeb} The sub-M\"obius structure
$M$
associated with any timed causal space
$(\ay,\cH,t)$
is M\"obius.
\end{thm}

\begin{proof} We show that 
$M$
satisfies equations (A) and (B) of Theorem~\ref{thm:submoeb_moeb}.
It suffices to check that for every unordered 5-tuple
$x,y,z,u,v\sub X$
of pairwise distinct points, equations (A) and (B) are satisfied
for some ordering 
$q\in\reg\cP_5$
of the 5-tuple, because in this case
$\de M(q)$
lies in an irreducible invariant subspace
$R$
of respective representation of
$S_5$,
describing M\"obius structures, see \cite{Bu2}. Hence,
$\de M(q)\in R$
for any ordering of the 5-tuple. Or, applying the procedure above,
to check equations (A) and (B) directly.
Thus we assume that
$q=(x,y,z,u,v)\in\reg\cP_5$
has the cyclic order
$xyzuv$.

Equation~(A) can be rewritten as
$$0=b_1+b_4-b_3+a_1=-c_1-b_3+b_4=:A(v)$$
because
$a_1+b_1+c_1=0$,
and we compute using (\ref{eq:time_(t4)}), (\ref{eq:time_(t6)})

\begin{align*}
A(v)=-c_1-b_3+b_4&=t_{34}^1+t_{12}^3-t_{15}^3+t_{23}^4-t_{12}^4\\ 
                      &=t_{34}^1-t_{15}^3+t_{23}^4-t_{12}^5\\
                      &=t_{34}^5+t_{34}^2-t_{15}^3+t_{23}^4-t_{12}^5\\
                      &=t_{34}^2-t_{15}^3+t_{23}^4\\
                      &=t_{34}^2+t_{23}^4-t_{15}^2-t_{15}^4\\
                      &=t_{34}^2-t_{15}^2=0 
\end{align*}

Similarly, Equation~(B) can be rewritten as
$0=b_1-b_2-a_4=:B(v)$
and we compute using (\ref{eq:time_(t4)}), (\ref{eq:time_(t6)})

\begin{align*}
B(v)=b_1-b_2-a_4&=t_{34}^1-t_{23}^1-t_{34}^2+t_{45}^2-t_{12}^4\\ 
                     &=t_{34}^5-t_{23}^1+t_{45}^2-t_{12}^4\\
                     &=t_{34}^5-t_{23}^1+t_{45}^1+t_{45}^3-t_{12}^4\\
                     &=t_{34}^5+t_{45}^3-t_{12}^4\\
                     &=t_{34}^5+t_{45}^3-t_{12}^3-t_{12}^5=0.
\end{align*}
Therefore, by Theorem~\ref{thm:submoeb_moeb},
$M$
is a M\"obius structure.
\end{proof}

\begin{pro}\label{pro:timed_monotone} The M\"obius structure
$M=\wh M(T)$
associated with a timed causal space
$T\in\cT$
is monotone, 
$M\in\cM$,
and the timed causal space
$\wh T(M)$
associated with 
$M$
coincide with
$T$, $\wh T(M)=T$.
\end{pro}

The proof proceeds in three steps, Lemmas~\ref{lem:pro:timed_monotone:axiom_m} --
\ref{lem:semi-metric_balls}.

\begin{lem}\label{lem:pro:timed_monotone:axiom_m} The M\"obius structure
$M=\wh M(T)$
satisfies Axiom~(M), and the time of the timed causal space
$T=(\ay,\cH,t)$
is computed in the usual way via
$M$-cross-ratios.
\end{lem}

\begin{proof} We check Axiom~(M) and simultaneously compute the time
$t(e,e')$
between events
$e$, $e'\in\ay$
assuming without loss of generality that 
$e=(y,y')$, $e'=(u,u')\in h_a$
for 
$a=(x,z)$
such that 
the 4-tuple
$q=(x,y,z,u)\in\reg\cP_4$
is obtained from
$\co(q)=xyuz$
by fixing the initial point 
$x$
and by the transposition of two last entries of
$\co(q)$.
Note that the pairs
$(x,u)$
and
$(y,z)$
separate each other. Then by definition,
$M(q)=(a(q),b(q),c(q))=(-t_{xy},t_{yu},t_{xy}-t_{yu})$
with the negative first entry
$a(q)=-t_{xy}$
and the positive second entry
$b(q)=t_{yu}$.
By Theorem~\ref{thm:submoeb_is_moeb}, we have 
$M(q)=M_d(q)$
for any semi-metric
$d\in M$.
Thus
$$\crr_1(q)=e^{a(q)}=\frac{d(x,z)d(y,u)}{d(x,u)d(y,z)}<1,\quad
  \crr_2(q)=e^{b(q)}=\frac{d(x,u)d(y,z)}{d(x,y)d(z,u)}>1.$$
This shows that 
$M$
satisfies Axiom~(M), see Remark~\ref{rem:equivalent_monotone}, and that 
$t(e,e')=t_{xz}=t_{yu}=\ln\crr_2(q)$.
\end{proof}

\begin{lem}\label{lem:pro:timed_monotone:harmonic} Let 
$h=h_e$
be a timelike line in a timed causal space
$T$.
An event
$d\in h_e$
if and only if the 4-tuple
$(d,e)$
is 
$M$-harmonic, that is, harmonic with respect to the M\"obius 
structure 
$M=\wh M(T)$.
\end{lem}

\begin{proof} Let
$e=(x,y)$
and
$d=(z,u)$.

If
$d\in h_e$,
then by Axiom~(h2),
$d$
separates
$e$,
and by Axiom~(t6) we have 
$t(y_a,u_a)=t(y_b,u_b)$,
where
$a=(x,z)$, $b=(x,u)$.
Thus we can assume without loss of generality that
$\co(q)=xzyu$
for the nondegenerate 4-tuple
$q=(e,d)$.
Note that 
$t(y_a,u_a)=t_{xz}$
and
$t(y_b,z_b)=t_{xu}$.
Thus
$t_{xz}=t_{xu}$.
The 4-tuple
$\wt q=(u,x,y,z)$
is obtained from
$\co(q)=uxzy$
by fixing the first entry
$u$
and permuting two last entries
$z$, $y$.
Therefore, by definition,
$M(\wt q)=(-t_{xu},t_{xz},0)$.
Using that 
$M(\wt q)=M_d(\wt q)$
for any semi-metric
$d\in M$,
we obtain
$$1=\crr_3(\wt q)=\frac{d(x,u)\cdot d(y,z)}{d(y,u)\cdot d(x,z)}.$$
Hence
$(d,e)$
is
$M$-harmonic.

Conversely, if
$(d,e)$
is
$M$-harmonic,
then by Lemma~\ref{lem:harm_separate},
$d$
and
$e$
separate each other. By Axiom~(h3), there is a unique
$u'\in X\sm e$
such that 
$d'=(z,u')\in h_e$.
By the first part of the proof, the 4-tuple
$(d',e)$
is
$M$-harmonic.
Taking a semi-metric
$d\in M$
with the infinitely remote point
$z$,
we observe that
\begin{equation}\label{eq:midpoints}
d(x,u)=d(u,y)\quad\textrm{and}\quad d(x,u')=d(u',y), 
\end{equation}
because the 4-tuples
$(d,e)$, $(d',e)$
are 
$M$-harmonic.
Assume
$u\neq u'$.
Then the 4-tuple
$(x,y,u,u')$
is nondegenerate, and
$u$, $u'$
are on the arc determined by
$e$
that does not contain
$z$.
Without loss of generality, we assume that
$(x,u)$
separate
$(y,u')$.
By Lemma~\ref{lem:pro:timed_monotone:axiom_m},
$M$
satisfies Axiom~(M). Thus
$$d(x,u)\cdot(y,u')>d(x,u')\cdot d(y,u)$$
in contradiction with (\ref{eq:midpoints}). Hence
$u=u'$
and 
$d=d'\in h_e$.
\end{proof}

\begin{lem}\label{lem:semi-metric_balls} The set 
$\cA$
of open arcs in
$X$
coincides with the set 
$\cB$
of open balls with respect to semi-metrics
$d\in M=\wh M(T)$
with infinitely remote points centered at finite points of
$d$, $\cA=\cB$.
\end{lem}

\begin{proof} Let
$\al\in\cA$
be an open arc in
$X$
and let
$x$, $y\in X$
be the end points of
$\al$.
We put
$e=(x,y)\in\ay$
and take 
$z\in\al$.
Then for 
$u=\rho_e(z)$
the event
$d=(z,u)$
lies on the timelike line
$h_e$.
By Lemma~\ref{lem:pro:timed_monotone:harmonic}, the 4-tuple
$(d,e)$
is
$M$-harmonic.
Thus
$z$
is the midpoint between
$x$, $y$
with respect to any semi-metric
$d\in M$
with infinitely remote point
$u$.
By Axiom~(M),
$v\in\al$
if and only if
$d(z,v)<r:=d(x,z)=d(y,z)$.
Therefore
$\al$
coincides with the open ball
$B_r(z)$
with respect to
$d$
of radius
$r$
centered at
$z$.
It means that 
$\cA\sub\cB$.

Let
$\be=B_r(o)\in\cB$
be the open ball with respect to a semi-metric
$\de\in M$
with the infinitely remote point
$\om$
of radius
$r>0$
centered at
$o\in X_\om$.
We show that
$\be\in\cA$.

Let
$d=(o,\om)\in\ay$.
By (h4), for any 
$y\in X_\om$, $y\neq o$,
there is a unique event
$e=y_d=(y,y')\in h_d$.
We fix such an
$y$,
denote by
$e^+\sub X$
the closed arc determined by
$e$
that contains
$\om$,
and consider the respective linear order
$<=<_e$
on
$h_d$
with the future arc 
$e^+$.

First,
we show that for 
$r=\de(y,o)$
the open ball
$B_r(o)$
coincides with the open arc 
$\intr e^-\in\cA$
determined by
$e$
that contains
$o$.
We denote by
$d^+\sub X$
the closed arc determined by
$d$
that contains
$y$,
by
$d^-$
the opposite closed arc. Then
$\intr e^-=(d^+\cap\intr e^-)\cup(d^-\cap\intr e^-)$.

We have
$u\in d^+\cap\intr e^-$
if and only if pairs
$(y,o)$, $(u,\om)$
separate each other. By Axiom~(M) this is equivalent to
$\de(u,o)<\de(y,o)=r$.
On the other hand,
$u\in d^-\cap\intr e^-$
if and only if
$u'=\rho_d(u)\in d^+\cap\intr e^-$.
By above, this is equivalent to
$\de(u',o)<r$.
By Lemma~\ref{lem:pro:timed_monotone:harmonic}, the 4-tuple
$(d,u_d)$
is harmonic, where
$u_d=(u,u')$.
Thus
$\de(u,o)=\de(u',o)<r$.
Therefore,
$\intr e^-=B_r(o)$
for 
$r=\de(y,o)$.

It remains to show that for any 
$r>0$
there is
$y\in X_\om$
with 
$\de(y,o)=r$.
We fix some
$y\in X_\om$, $y\neq o$,
and use the notations introduced above.
By (t4b), for any 
$s>0$
there is
$e_\pm=(u_\pm,u_\pm')\in h_d\cap C_e^\pm$
with 
$t(e,e_\pm)=s$.
By Lemma~\ref{lem:pro:timed_monotone:harmonic}, the 4-tuples
$(d,e)$, $(d,e_\pm)$
are 
$M$-harmonic. Hence
$\de(o,y)=\de(o,y')$, $\de(o,u_\pm)=\de(o,u_\pm')$.
As above, Axiom~(M) implies
$$\de(u_-,o)<\de(y,o)<\de(u_+,o).$$
By Lemma~\ref{lem:pro:timed_monotone:axiom_m}, the time 
$t(e,e_\pm)$
is computed via 
$M$-cross-ratios,
$$t(e,e_\pm)=\left|\ln\frac{\de(\om,y)\de(u_\pm,o)}{\de(\om,u_\pm)\de(y,o)}\right|
            =\left|\ln\frac{\de(u_\pm,o)}{\de(y,o)}\right|,$$
hence
$s=\pm\ln\frac{\de(u_\pm,o)}{\de(y,o)}$.
This shows that for any 
$\la>0$
there is
$u\in X_\om$, $u\neq o$,
with 
$\de(u,o)=\la \de(y,o)$.  
Hence, for any 
$r>0$
there is
$y\in X_\om$
with 
$\de(y,o)=r$.
\end{proof}

\begin{proof}[Proof of Proposition~\ref{pro:timed_monotone}]
By Lemma~\ref{lem:pro:timed_monotone:axiom_m}, the M\"obius structure
$M=\wh M(T)$
satisfies Axiom~(M) for any timed causal space
$T=(\ay,\cH,t)\in\cT$. 
It follows from Lemma~\ref{lem:semi-metric_balls} that 
$M$
satisfies Axiom~(T). Thus 
$M$
is monotone,
$M\in\cM$.

Let
$T'=(\ay,\cH',t')=\wh T(M)\in\cT$
be the timed causal space determined by
$M$.
By Lemma~\ref{lem:pro:timed_monotone:harmonic}, 
$\cH'=\cH$,
and by Lemma~\ref{lem:pro:timed_monotone:axiom_m},
$t'=t$.
Thus
$T'=T$. 
\end{proof}

\begin{proof}[Proof of Theorem~\ref{thm:main}]
Given 
$M\in\cM$,
we show that
$M'=M$,
where
$M'=\wh M\circ\wh T(M)$,
that is,
$M'(q)=M(q)$
for every
$q\in\reg\cP_4$.
Using Eq.~\ref{eq:signed_cross-ratio_homomorphism}, we can assume
without loss of generality that the cyclic order of
$q=(x,y,z,u)$
is
$\co(q)=xyuz$,
and thus
$q$
is obtained from
$\co(q)$
by picking up the first entry
$x$
and permuting the last two entries. In particular, 
$(x,u)$
and 
$(y,z)$
separate each other. Then by definition~(\ref{eq:def_moeb_timed})
we have 
$$M'(q)=(-t_{xy},t_{yu},t_{xy}-t_{yu}),$$
where 
$t_{xy}=t(z_a,u_a)$, $t_{yu}=t(x_b,z_b)$
for 
$a=(x,y)$, $b=(y,u)\in\ay$,
see Axiom~(h4),
for 
$T=\wh T(M)=(\ay,\cH,t)\in\cT$.
By definition~(\ref{eq:time}) of the time
$t$
we have
$t(z_a,u_a)=\left|\ln\frac{|xz|\cdot|yu|}{|xu|\cdot|yz|}\right|=-\ln\crr_1(q)$,
$t(x_b,z_b)=\left|\ln\frac{|xu|\cdot|yz|}{|xy|\cdot|uz|}\right|=\ln\crr_2(q)$
(to choose the signs, we have used that 
$(x,u)$, $(y,z)$
separate each other and monotonicity of
$M$).
Therefore,
$M'(q)=M(q)$.
Together with Proposition~\ref{pro:timed_monotone} this shows that
$\wh T:\cM\to\cT$
and 
$\wh M:\cT\to\cM$
are mutually inverse maps.

Let
$\wh g:\ay\to\ay$
be an automorphism of some 
$T=(\ay,\cH,t)\in\cT$.
Since
$t(e,e')=0$
if and only if the events
$e$, $e'\in\ay$
lie on a light line, and 
$\wh g$
preserves the time 
$t$,
we see that 
$\wh g$
maps every light line to a light line. Thus
$\wh g$
determines a map 
$g:X\to X$
with
$\wh g(p_x)=p_{g(x)}$,
see sect.~\ref{subsect:recoveting_desitter}. For any event
$e=(x,y)\in\ay$
we have
$e=p_x\cap p_y$.
Thus 
$\wh g(e)=\wh g(p_x)\cap\wh g(p_y)=p_{g(x)}\cap p_{g(y)}=(g(x),g(y))$.
Hence,
$\wh g$
is induced by
$g$.

Since 
$T=\wh T(M)$
for some 
$M\in\cM$,
the timelike lines and the time of
$T$
are determined by cross-ratios of
$M$,
see Proposition~\ref{pro:monotone_moeb_to_timed_space}. Therefore,
$g$
is an 
$M$-automorphism.
If
$g=\id_X$,
then
$\wh g=\id_{\ay}$.
Thus the group
$G_T$
of
$T$-automorphisms
injects into the group
$G_M$
of
$M$-automorphisms.
Together with Proposition~\ref{pro:monotone_moeb_to_timed_space}
this shows that the groups
$G_M$
and 
$G_T$
are canonically isomorphic.
\end{proof}

\section{Time inequalities} 
\label{sect:time_inequality}

The {\em time inequality} for de Sitter 2-space 
$\ds^2$
says that  
$$t(a,b)+t(b,c)\le t(a,c)$$
for any events
$a<b<c$
with the equality in the case
$t(a,c)>0$
if and only if
$a,b,c$
are events on a timelike line. We first show in sect.~\ref{subsect:time_inequality_ds}
that this inequality follows from properties of Lambert quadrilaterals.
Then in sect.~\ref{subsect:hierarchy}, we discuss a hierarchy of time conditions, 
which includes the time inequality,
and show that every timed causal space
$T\in\cT$
satisfies the {\em weak time inequality,} see Theorem~\ref{thm:wti}.
In sect.~\ref{subsect:monotone_vp} we describe monotone M\"obius structures 
which satisfy Variational Principle, (VP), the most strong time condition 
from the list, and in sect.~\ref{subsect:convex_moeb} also {\em convex} 
M\"obius structures. We show that these two classes contain the canonical M\"obius structure,
and that the first one contains a neighborhood of the canonical structure in a fine
topology.

\subsection{The time inequality for $\ds^2$ via $\hyp^2$}
\label{subsect:time_inequality_ds}

The time inequality for de Sitter 2-space
$\ds^2$
follows from properties of Lambert quadrilaterals in
$\hyp^2$.
This goes of course via the canonical M\"obius structure 
$M_0$
on the common absolute
$S^1$.
More precisely, we use the fact that harmonicity
of a 4-tuple
$((x,y),(z,u))\sub S^1$
with respect to
$M_0$
is equivalent to orthogonality of the geodesics
$xy$, $zu\sub\hyp^2$.

Recall that a {\em Lambert quadrilateral} 
$\al\be\ga o$
in the hyperbolic plane
$\hyp^2$ 
has three right angles at
$\al$, $\be$
and
$\ga$.
The fours angle at
$o$
is acute, and
$|\al\be|<|o\ga|$, $|\be\ga|<|\al o|$.
Now, we explain how these properties imply the time inequality for 
$\ds^2$.

\begin{figure}[htbp]
\centering
\psfrag{a}{$a$}
\psfrag{b}{$b$}
\psfrag{c}{$c$}
\psfrag{h}{$d$}
\psfrag{hpr}{$p$}
\psfrag{hprpr}{$q$}
\psfrag{o}{$o$}
\psfrag{al}{$\al$}
\psfrag{alpr}{$\al'$}
\psfrag{be}{$\be$}
\psfrag{bepr}{$\be'$}
\psfrag{ga}{$\ga$}
\psfrag{gapr}{$\ga'$}
\includegraphics[width=0.6\columnwidth]{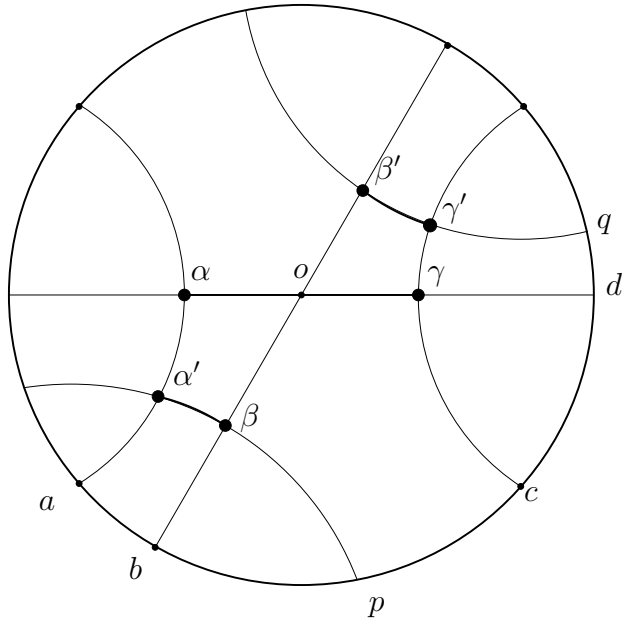}
\caption{The time inequality in $\ds^2$}\label{fi:time}
\end{figure}

Let
$a$, $b$, $c$
be events in
$\ay$
such that 
$a<b<c$
for the order
$<:=<_b$.
We consider a generic case when no pair of events
$(a,b)$, $(b,c)$
lies on a light line, and the events are not on a common
timelike line. Then there are events
$d$, $p$, $q\in\ay$
with
$a,b\in h_p$, $a,c\in h_d$, $b,c\in h_q$.
We pass to the 
$\hyp^2$-picture,
and draw the respective timelike lines as geodesics in
$\hyp^2$
with the same ends on the absolute
$S^1$.
Since the time in
$\ds^2$
and the distance in
$\hyp^2$
are computed via cross-ratios with respect to
$M_0$,
we have 
$t(a,c)=|\al\ga|$, $t(a,b)=|\al'\be|$, $t(b,c)=|\be'\ga'|$,
see Figure~\ref{fi:time}.

But
$|\al\ga|=|\al o|+|o\ga|$,
and quadrilaterals
$\al\al'\be o$, $\ga\ga'\be' o$
have right angles at
$\al$, $\al'$, $\be$
respectively at
$\ga$, $\ga'$, $\be'$,
i.e., they are Lambert quadrilaterals. Thus
$|\al o|>|\al'\be|$, $|o\ga|>|\be'\ga'|$,
and we obtain
$$t(a,c)>t(a,b)+t(b,c).$$

\subsection{Hierarchy of time conditions}
\label{subsect:hierarchy}

We assume that a timed causal space
$T=\{\ay,\cH,t\}\in\cT$
is fixed together with the respective monotone M\"obius structure
$M=\wh M(T)\in\cM$.

We say that an event
$b\in\ay$
is {\em strictly between} events
$a$
and
$c\in\ay$
if
$a$
and
$c$
lie on different open arcs in
$X$
defined by
$b$.
Note that in this case,
$a$, $b$, $c$
are pairwise in the strong causal relation, in particular,
$a<b<c$ 
for appropriately chosen
$<:=<_b$. 

Let
$a=(o,o')$, $b=(\om,\om')\in\ay$
be events in the strong causal relation such that the pairs
$(o,\om')$
and 
$(o',\om)$
separate each other. Then
$(o,\om),(o',\om')$
are also in the strong causal relation. Let
$d=(x,x')\in\ay$
be an event strictly between
$(o,\om)$
and 
$(o',\om')$.
We denote by
$$t_d^+(a,b)=t(o_d,\om_d),\quad t_d^-(a,b)=t(o_d',\om_d').$$
In general,
$t_d^+(a,b)\neq t_d^-(a,b)$.
However, if
$a,b\in h_d$,
then 
$t_d^+(a,b)=t_d^-(a,b)=t(a,b)$
by definition of
$t(a,b)$,
see (\ref{eq:time}), (\ref{eq:time_different}) and 
Lemma~\ref{lem:unique_common_perpendicular}. 

We consider the function
$$F_{ab}(d)=\frac{1}{2}(t_d^+(a,b)+t_d^-(a,b))$$
on the set 
$D_{ab}$
of events
$d\in\ay$
that are strictly between
$(o,\om)$
and 
$(o',\om')$,
and introduce the following list of time conditions for 
$T$
and therefore simultaneously for 
$M$.

\begin{itemize}
 \item [(VP)] Variational principle: the infimum of
$F_{ab}$
is taken  at unique
$d_0\in D_{ab}$
for which
$a,b\in h_{d_0}$;
 \item[(LQI)] Lambert quadrilateral inequality:
$$F_{ab}(d)>F_{ab}(d_0)$$
for every
$d\in D_{ab}\sm d_0$
such that 
$a\in h_d$;
 \item[(TI)] Time inequality: 
$$t(a,b)+t(b,c)\le t(a,c)$$
for any 
$a<b<c$
with the equality in the case
$t(a,c)>0$
if and only if
$a,b,c$
are events on a timelike line;
\item[(WTI)] Weak time inequality:
$$t(a,b)+t(b,c)<t(a,c)$$
for any
$a<b<c$
such that 
$b$
lies on a light line either with
$a$
or with
$c$,
and 
$a$, $c$
are not on a light line.
\end{itemize}

We have the following implications
$$\mathrm{(VP)}\Rightarrow\mathrm{(LQI)}\Rightarrow\mathrm{(TI)}\Rightarrow\mathrm{(WTI)}.$$
The first and the last implications are obvious, and we explain the second
implication in Proposition~\ref{pro:lqi_ti}. For the canonical
M\"obius structure
$M_0$,
the geometric meaning of the function
$F_{ab}:D_{ab}\to\R$
is especially clear.

\begin{pro}\label{pro:functional_canonical} Let
$a=(o,o')$, $b=(\om,\om')\in\ay$
be events in the strong causal relation such that the pairs
$(o,\om')$
and 
$(o',\om)$
separate each other,
$d=(x,x')\in D_{ab}$.
Then for the canonical M\"obius structure
$M_0$,
the value
$F_{ab}(d)$
is the distance in
$\hyp^2$
between the points 
$p=oo'\cap xx'$, $q=\om\om'\cap xx'$
at which the geodesic
$xx'\sub\hyp^2$
intersects the geodesics
$oo'$
and 
$\om\om'$.
\end{pro}

\begin{proof} Using that the time between events in a timed causal
space and the distance in
$\hyp^2$
are computed via the respective cross-ratios, we see that
$t_d^+(a,b)=t(o_d,\om_d)$
is distance in
$\hyp^2$
between projections 
$\wh o$, $\wh\om$
of
$o$, $\om$
to the geodesic
$xx'\sub\hyp^2$,
and similarly,
$t_d^-(a,b)=t(o_d',\om_d')$
is distance between projections 
$\wh o'$, $\wh\om'$
of
$o'$, $\om'$
to the same geodesic
$xx'$.
By the angle parallelism formula,
$\wh op=p\wh o'$
and 
$\wh\om q=q\wh\om'$.
Therefore,
$F_{ab}(d)=\frac{1}{2}(|\wh o\wh\om|+|\wh o'\wh\om'|)=|pq|$. 
\end{proof}

\begin{cor}\label{cor:vp_canonical} The canonical M\"obius structure
satisfies (VP).
\end{cor}

\begin{proof} This immediately follows from properties of the distance in
$\hyp^2$
between points on geodesics. 
\end{proof}

\subsection{The weak time inequality}
\label{subsect:wti}

Here, we prove the following

\begin{thm}\label{thm:wti} Any timed causal space
$T=\{\ay,\cH,t\}\in\cT$
satisfies (WTI).
\end{thm}

\begin{lem}\label{lem:order_harmonic} Assume distinct events
$a$, $b$
lie on a common timelike line,
$a,b\in h_c$,
where
$a=(x,y)$, $b=(z,u)$, $c=(v,w)\in\ay$,
and suppose that 
$v$
lies on the open arc 
$\ga$
between
$x$, $y$
that does not contain
$b$.
Then for every
$s\in\ga$, $s\neq v$,
and for 
$d=(s,t)\in h_a$, $d'=(s,t')\in h_b$
we have:
$t'$
lies on the open arc 
$\si$
between
$w$, $t$
that does not contain
$s$.
\end{lem}

\begin{proof} Moving
$s$
along
$\ga$,
observe that for 
$s=v$
we have
$t=t'$,
while for 
$s$
approaching to
$x$
or
$y$
the point
$t$
is not on the arc between
$z$, $u$
that contains
$w$.
Therefore
$t'$
lies on 
$\si$
for these extremal cases. By continuity of reflections
$\rho_a$, $\rho_b:X\to X$
and Lemma~\ref{lem:unique_common_perpendicular} we have
$t'\in\si$
for every
$s\in\ga$.
\end{proof}

\begin{lem}\label{lem:monotone_plus_dist} Let
$a=(o,o')$, $b=(\om,\om')\in\ay$
be events in the strong causal relation such that the pairs
$(o,\om')$
and 
$(o',\om)$
separate each other. Then the function
$F_{ab}^+(d)=t(o_d,\om_d)$
is monotone on the set 
$D_{ab}$
of events
$d\in\ay$
that are strictly between
$e=(o,\om)$
and 
$e'=(o',\om')$, $F_{ab}^+(d)<F_{ab}^+(d')$
for any 
$d$, $d'\in D_{ab}$
with 
$d<d'<e$.
\end{lem}

\begin{proof} Let 
$d=(x,y)$.
By Axiom~(t5) we have 
$t(o_d,\om_d)=t(x_e,y_e)$.
Thus
$F_{ab}^+(d)=t(x_e,y_e)$.
For 
$d'=(x',y')$
between
$d$
and 
$e$,
the segment
$x_ey_e\sub h_e$
is contained in the segment
$x_e'y_e'\sub h_e$
and does not coincide with it (though, we do not exclude a possibility 
that these segments have a common end). Thus
$t(x_e,y_e)<t(x_e',y_e')$. 
\end{proof}

\begin{proof}[Proof of Theorem~\ref{thm:wti}] Let
$a,b,c\in\ay$
be events in the causal relation,
$a<b<c$,
and we assume without loss of generality that 
$b$, $c$
are on a light line. Then
$t(b,c)=0$,
and the required inequality is reduced to
$t(a,b)<t(a,c)$.
We assume furthermore without loss of generality that 
$a=(o,o')$, $b=(\om,\om')$
and the pairs 
$(o,\om')$, $(o',\om)$
separate each other. Since
$b$, $c$
are on a light line, we can assume that
$c=(\om,\om'')$.
Then the assumption
$a<b<c$
implies that 
$\om''$
is on the (open) arc 
$\al$
between
$\om$, $\om'$
that does not contain
$a$.

There is
$d=(x,y)\in\ay$
with
$a,b\in h_d$.
We assume that 
$x$
is on the arc 
$\be$
between
$o$, $o'$
that does not contain
$b$.
Then
$y\in\al$.
Similarly, there is
$d'=(x',y')\in\ay$
with 
$a,c\in h_{d'}$.
We also assume that 
$x'$
is on the arc 
$\be$.
Then
$y'$
is on the arc 
$\al'\sub\al$
between
$\om$, $\om''$. 
Note that
$d$, $d'\in D_{ab}$
and that 
$d\neq d'$
since 
$b\neq c$,
therefore
$x'\neq x$
because
$d$, $d'\in h_a$.

We claim that 
$x'$
lies on the arc 
$\be$
between
$x$
and 
$o$.
Indeed, since
$d$, $d'\in h_a$,
we otherwise would have by Lemma~\ref{lem:order_harmonic}, substituting 
$o$
for 
$v$, $o'$
for 
$w$,
$\om$
for 
$s$,
$\om'$
for 
$t$,
$\om''$
for 
$t'$,
that 
$\om''\not\in\al$
in contradiction with the previously established 
$\om''\in\al$.

It follows that 
$d<d'<e=(o,\om)$.
By Lemma~\ref{lem:monotone_plus_dist},
$F_{ab}^+(d)<F_{ab}^+(d')$.
On the other hand,
$F_{ab}^+(d)=t(o_d,\om_d)=t(a,b)$
and 
$F_{ab}^+(d')=t(o_{d'},\om_{d'})=t(a,c)$.
\end{proof}

\subsection{Implication (LQI)$\Longrightarrow$(TI)}
\label{subsect:lqi_ti}

Here, we show that the Lambert quadrilateral inequality implies the 
time inequality.

\begin{pro}\label{pro:lqi_ti} 
$\mathrm{(LQI)}\Rightarrow\mathrm{(TI)}$. 
\end{pro}

Assume
$b\in\ay$
is strictly between
$a,c\in\ay$.
Then by Corollary~\ref{cor:common_perpendicular_prescribed}, 
there are common perpendiculars
$p$
to
$a,b$
and 
$q$
to
$b,c$.

\begin{lem}\label{lem:common_perp_separation} Assume 
$a$, $c\in h_d$
and 
$b\in\ay\sm h_d$
is strictly between
$a$
and
$c$.
Then
$d$
is strictly between the common perpendiculars 
$p$
to
$a$, $b$
and 
$q$
to
$b$, $c$.
\end{lem}

\begin{proof} By the assumption,
$b$
is not on the timelike line
$h_d\sub\ay$.
Hence, the common perpendicular
$p=(p',p'')\in\ay$
to
$a$, $b$
is not equal to
$d$, $p\neq d$,
and the common perpendicular
$q=(q',q'')\in\ay$
to
$b$, $c$
is not equal to
$d$, $q\neq d$.

Since
$p$, $d\in h_a$,
the events
$p$, $d$
are not on a light line, and some closed arc in
$X$
determined by
$d$
does not include
$p$.
We denote that arc by
$d^+\sub X$.
Hence,
$p<_dd$
for the respective partial order
$<_d$.

We also denote by
$b^+\sub X$
the closed arc determined by
$b$
that includes
$c$.
Without loss of generality, we assume that 
$p''$, $w$, $q'\in b^+$,
where
$d=(v,w)$.
Then by Lemma~\ref{lem:order_harmonic}, applied to
$a$, $b\in h_p$ 
and 
$\ga=b^+$
we see that
$v$
lies on the arc 
$\si$
determined by
$(p',t)$
that does not contain
$w$,
where
$r=(t,w)$
is orthogonal to 
$b$.
Therefore,
$d<_dr$
and 
$t\in d^+$.

We denote by
$r_d^+$
the closed arc in
$X$
determined by
$r$
that does not include
$d$,
see sect.~\ref{subsect:causality_structure}. Since
$d$
is also orthogonal to
$c$,
applying again Lemma~\ref{lem:order_harmonic} to
$b$, $c\in h_q$, 
we see that
$t$
lies on the arc 
$\si'$
determined by 
$(v,q'')$
that does not contain
$q'\in d^+$.
Since 
$r$, $q\in h_b$,
it means that 
$q\sub r_d^+$,
thus
$r<_dq$.

Therefore,
$p<_dd<_dr<_dq$.
Since by construction,
$p$, $r$, $q\in h_b$,
and 
$p$, $d$
are not on a light line, we see that 
$d$
is strictly between
$p$
and
$q$.  
\end{proof}

\begin{cor}\label{cor:common_perp_separation} Assume
$a$, $c\in h_d$
as in Lemma~\ref{lem:common_perp_separation}. Given
$p$, $q\in\ay$
with
$p\perp a$, $q\perp c$
such that
$d$
is strictly between
$p$ 
and
$q$,
the common perpendicular to
$p$, $q$
is strictly between
$a$
and
$c$. 
\end{cor}

\begin{proof} Since
$d$
is strictly between
$p$
and  
$q$,
the events
$p$, $q$
are in the strong causal relation. Thus their common perpendicular
$b\in\ay$
exists and is uniquely determined by Corollary~\ref{cor:common_perpendicular_prescribed}
and Lemma~\ref{lem:unique_common_perpendicular}. Since
$a$
is the common perpendicular to
$d$, $p$,
and
$c$
is the common perpendicular to
$d$, $q$,
Lemma~\ref{lem:common_perp_separation} implies that 
$b$
is strictly between
$a$
and
$c$.
\end{proof}

\begin{proof}[Proof of Proposition~\ref{pro:lqi_ti}] Assume
$a<b<c$
for events in
$\ay$.
If
$t(a,c)=0$
then by Axiom~(t2),
$a$, $c$
are on a light line,
$a,c\in p_x$
for some
$x\in X$.
Then
$b\in p_x$,
and we have 
$t(a,b)=t(b,c)=t(a,c)=0$.

Therefore, we can assume that
$t(a,c)>0$
and hence
$a,c\in h_d$
for some timelike line 
$h_d\sub\ay$.
Using Theorem~\ref{thm:wti}, we also can assume that 
$b$
lies on a light line neither with
$a$
nor with
$b$.
If
$b$
is also on
$h_d$,
then by Axiom~(t4a),
$t(a,b)+t(b,c)=t(a,c)$.
To complete the proof, we show that the assumption
$b\not\in h_d$
implies the strict inequality in the time inequality.
In this case,
$b$
is strictly between
$a$, $c$
by our assumption, and there are 
$p$, $q\in\ay$
with
$a,b\in h_p$, $b,c\in h_q$.
By Lemma~\ref{lem:common_perp_separation},
$d$
is strictly between
$p$
and 
$q$.  

Since
$a\in h_p$
and 
$p$, $d\in D_{ab}$,
(LQI) applied to
$a$, $b$
gives
$F_{ab}(d)>F_{ab}(p)$.
Since
$c\in h_q$
and 
$q$, $d\in D_{bc}$,
(LQI) applied to
$b$, $c$
gives
$F_{bc}(d)>F_{bc}(q)$.
On the other hand,
$F_{ab}(p)=t(a,b)$, $F_{bc}(q)=t(b,c)$,
and it remains to show that 
$F_{ab}(d)+F_{bc}(d)=t(a,c)$.

We fix decomposition
$X=d^+\cup d^-$, $d^+\cap d^-=d$,
induced by
$d$,
and write
$a=(a^+,a^-)$, $b=(b^+,b^-)$, $c=(c^+,c^-)$,
where
$a^\pm,b^\pm,c^\pm\in d^\pm$.
By (t4a), we have
$t(a_d^\pm,b_d^\pm)+t(b_d^\pm,c_d^\pm)=t(a_d^\pm,c_d^\pm)$.
Therefore
$F_{ab}(d)+F_{bc}(d)=\frac{1}{2}(t(a_d^+,c_d^+)+t(a_d^-,c_d^-))=t(a,c)$
because
$a,c\in h_d$. 
\end{proof}

\begin{cor}\label{cor:vp_ti} Variational principle implies the time
inequality, 
(VP)$\Longrightarrow$(TI),
cp. \cite{PY}.
\end{cor}

\subsection{Monotone M\"obius structures with (VP)}
\label{subsect:monotone_vp}

Some important properties of M\"obius structures
$\cM$
which do not follow from monotonicity Axiom~(M) can be expressed
as an inequality
$\crr(q)>\crr(q')$
between cross-ratios of 4-tuples
$q$, $q'$
with two common entries,
$|q\cap q'|=2$,
under an assumption that a symmetry between
$q$, $q'$
is broken down in a definite way.

We use notation
$\reg\cP_n$
for the set of ordered nondegenerate
$n$-tuples
of points in
$X=S^1$, $n\in\N$.
For 
$q\in\reg\cP_n$
and a proper subset
$I\sub\{1,\dots,n\}$
we denote by
$q_I\in\reg\cP_k$, $k=n-|I|$,
the 
$k$-tuple
obtained from
$q$
(with the induced order) by crossing out all entries which correspond to elements of
$I$.

We introduce the following Axiom for a M\"obius structure
$M\in\cM$,
which implies the variational principle (VP).

(I) Increment: for any 
$q\in\reg\cP_7$
with cyclic order
$\co(q)=1234567$
such that 
$q_{247}$
and 
$q_{157}$
are harmonic, we have 
$$\crr_1(q_{345})>\crr_1(q_{123}).$$

It means the following. Assume we are given two events
$e=(o,\om)$, $e'=(o',\om')\in\ay$
in the strong causal relation such that 
$(o,\om')$
and 
$(o',\om)$
separate each other. Let
$oo'\sub X$
be the arc between 
$o$, $o'$
that does not contain
$\om$, $\om'$,
and let 
$(u,v)\in\ay$, $u\in oo'$,
be the common perpendicular to
$a=(o,o')$, $b=(\om,\om')$,
that is 
$(u,v)\in h_a\cap h_b$.
Given
$x\in oo'$
such that 
$(o,u)$
and 
$(o',x)$
separate each other, we put
$g_+(u,x)=\exp t_e(u_e,x_e)$, $g_-(u,x)=\exp(-t_{e'}(u_{e'},x_{e'}))$,
$\De(u,x)=g_+(u,x)g_-(u,x)$.
Then Axiom~(I) tells that
$\De(u,x)>1$.

Indeed, consider
$q=(o,\om,v,\om',o',u,x)\in\reg\cP_7$
written in the cyclic order
$\co (q)=1234567$.
The assumption that 4-tuples
$q_{247}$
and 
$q_{157}$
are harmonic means that 4-tuples
$(u,o,v,o')$
and 
$(u,\om,v,\om')$
are harmonic with the common axis
$(u,v)$,
i.e.,
$(u,v)\in h_a\cap h_b$.
Since
$q_{345}=(o,\om,u,x)$, $q_{123}=(\om',o',u,x)$,
we have 
$g_+(u,x)=\crr_1(q_{345})$, $g_-(u,x)=1/\crr_1(q_{123})$.
Thus the condition 
$\crr_1(q_{345})>\crr_1(q_{123})$
means that 
$\De(u,x)>1$.

\begin{pro}\label{pro:canonical_vp} The canonical M\"obius structure
$M_0$
on
$X$
satisfies Axiom~(I).
\end{pro}

\begin{proof} Let
$q=(o,\om,v,\om',o',u,x)\in\reg\cP_7$
be as above. In the metric on
$X$
from
$M_0$
with infinitely remote point
$u$,
we have
$|vo|=|vo'|$, $|v\om|=|v\om'|$.
Since
$M_0$
is canonical,
$|vo|=|v\om|+|o\om|$,
thus
$|o\om|=|o'\om'|$.
Furthermore,
$\crr_1(q_{345})=\crr_1(o,\om,u,x)=|x\om|/|ox|$
and 
$\crr_1(q_{123})=\crr_1(\om',o',u,x)=|xo'|/|x\om'|$.

Note that
$xo\sub x\om'\sub X_u$.
Thus
$|xo|<|x\om'|$.
Using
$|x\om|=|xo|+|o\om|$
and 
$|xo'|=|x\om'|+|o'\om'|=|x\om'|+|o\om|$,
we obtain
$|x\om|/|ox|>|xo'|/|x\om'|$.
Hence, 
$\crr_1(q_{345})>\crr_1(q_{123})$,
and  
$M_0$
satisfies (I). 
\end{proof}

\begin{pro}\label{pro:increment_variational_principle} Increment Axiom~(I)
implies Variational Principle~(VP).
\end{pro}

\begin{proof} Let
$a=(o,o')$, $b=(\om,\om')\in\ay$
be events in the strong causal relation such that the pairs
$(o,\om')$
and 
$(o',\om)$
separate each other. Then the events 
$e=(o,\om)$, $e'=(o',\om')$
are also in the strong causal relation.

Let
$d_0=(u,v)\in D_{ab}$
be the unique event with
$a$, $b\in h_{d_0}$.
We show that
$F_{ab}(d)>F_{ab}(d_0)$
for any 
$d=(x,x')\in D_{ab}$, $d\neq d_0$.
Let 
$oo'\sub X$
be the arc between
$o$, $o'$
that does not include
$b$.
Without loss of generality we can assume that
$u,x\in oo'$
and 
$x\neq u$.
It suffices to show that 
$F_{ab}(d)>F_{ab}(d')$
for 
$d'=(u,x')$.

Let
$\si\sub h_e$
be the segment between
$u_e$, $x_e'\in h_e$, $\si'\sub h_{e'}$
the segment between
$u_{e'}$, $x_{e'}'\in h_{e'}$.
Since 
$x\neq u$,
one of the events
$x_e\in h_e$, $x_{e'}\in h_{e'}$
lies in the respective segment
$\si$, $\si'$,
while the other not. We assume without loss of generality that 
$x_{e'}\in\si'$.
Then 
$x_e\not\in\si$,
and moreover
$u_e$
separates the events
$x_e$
and 
$x_e'$
on the timelike line 
$h_e$.
Thus
$t(x_e,x_e')>t(u_e,x_e')$
while
$t(x_{e'},x_{e'}')<t(u_{e'},x_{e'}')$.
By Axiom~(I),
$t(x_e,u_e)>t(x_{e'},u_{e'})$,
thus
$t(x_e,x_e')-t(u_e,x_e')>t(u_{e'},x_{e'}')-t(x_{e'},x_{e'}')$.

Recall that 
$$F_{ab}(d)=\frac{1}{2}(t_d^+(a,b)+t_d^-(a,b)),$$
where
$t_d^+(a,b)=t(o_d,\om_d)$, $t_d^-(a,b)=t(o_d',\om_d')$. 
By (t5) we have 
$t(o_d,\om_d)=t(x_e,x_e')$, $t(o_d',\om_d')=t(x_{e'},x_{e'}')$.
Hence
$$F_{ab}(d)-F_{ab}(d')=\frac{1}{2}(t(x_e,x_e')-t(u_e,x_e')
                       +t(x_{e'},x_{e'}')-t(u_{e'},x_{e'}'))>0,$$
which completes the proof.
\end{proof}

Using Corollary~\ref{cor:vp_ti}, we immediately obtain

\begin{cor}\label{cor:i_ti} Increment Axiom~(I) implies the time 
inequality, (TI).
\end{cor}

\subsection{The fine topology and Axiom (I)}
\label{subsect:fine_topology}

We denote by
$\cI$
the class of monotone M\"obius structures on the circle
which satisfies Axiom~(I). This work does not provide tools,
which allow to answer natural questions like to 
characterize hyperbolic spaces
$Y$
with
$\di Y=S^1$
for which the respective M\"obius structure is in the class
$\cI$.
We only show here that a neighborhood of the canonical M\"obius structure
$M_0$
on
$X=S^1$
in an appropriate topology lies in
$\cI$.

Recall that a M\"obius structure 
$M$
on a set 
$X$
determines the 
$M$-topology
on
$X$
(see sect.~\ref{subsect:semi-metrics_topology})
and hence the induced topology on the set 
$\reg\cP_n(X)\sub X^n$. 
One can consider a M\"obius structure as a map defined on
$\reg\cP_4$
with values in a vector space (see sect.~\ref{subsect:moeb_sub-moeb}).
Thus it not clear how to define a topology on a set of M\"obius 
structures on
$X$
because the topology of
$X$
may change together with change of a M\"obius structure.

However, for monotone M\"obius structures on
$X=S^1$
such a problem does not exist in view of Axiom~(T): all M\"obius
structures
$M\in\cM$
induce on
$X$
one and the same topology of the circle. We define a fine topology on
$\cM$
as follows.

Let
$\reg^+\cP_7\sub X^7$
be the subset of
$\reg\cP_7$
which consists of all 
$q\in\reg\cP_7$
with the cyclic order. That is, for 
$q\in\reg^+\cP_7$
we have
$\co(q)=q$.
We take on
$\reg^+\cP_7$
the topology induced from the standard topology of the 7-torus
$X^7$.
We associate with a M\"obius structure
$M\in\cM$
a section of the trivial bundle
$\reg^+\cP_7\times\R^4\to\reg^+\cP_7$
given by
$$M(q)=(q,\crr_2(q_{247}),\crr_2(q_{157}),\crr_1(q_{345}),\crr_1(q_{123}))$$
for 
$q=1234567\in\reg^+\cP_7$.
Taking the product topology on
$\reg^+\cP_7\times\R^4$,
we define the {\em fine} topology on
$\cM$
with base given by sets
$$U_V=\set{M\in\cM}{$M(\reg^+\cP_7)\sub V$},$$
where
$V$
runs over open subsets of
$\reg^+\cP_7\times\R^4$.

We show that the canonical M\"obius structure
$M_0$
on
$X$
possesses a neighborhood
$U_V$
in the fine topology which lies in
$\cI$,
that is, every M\"obius structure
$M\in U_V$
satisfies Axiom~(I). To this end, consider a function
$\ep:\reg^+\cP_7\to\R$
given by
$$\ep(q)=\frac{|o\om|_0^2}{4|x\om'|_0^2}$$
for 
$q=(o,\om,v,\om',o',u,x)\in\reg^+\cP_7$,
where
$|\cdot\,\cdot\,|_0$
is a standard metric on
$X_u=\R$
from the canonical M\"obius structure
$M_0$
with infinitely remote point
$u$.
Such a metric is determined up to a homothety, but clearly
$\ep$
does not depend on that.

\begin{lem}\label{lem:ep_continuous} The function
$\ep:\reg^+\cP_7\to\R$
is continuous.
\end{lem}

\begin{proof} Obviously, it suffices to check that 
$\ep$
varies continuously in the variable 
$u\in q$.
We switch to the notation
$d_u(x,y)=|xy|_0$
for a metric from
$M_0$
with infinitely remote point
$u$.
Applying to
$u'\in X$, $u'\neq u$,
a metric inversion, we have
$$d_{u'}(x,y)=\frac{d_u(x,y)}{d_u(u',x)d_u(u',y)}.$$
The point
$u'\in X$
is infinitely remote for 
$d_{u'}$.
Thus for 
$q'=(o,\om,v,\om',o',u',x)$, $q=(o,\om,v,\om',o',u,x)$
we obtain
$$\ep(q')=\frac{d_{u'}^2(o,\om)}{4d_{u'}^2(x,\om')}
         =\ep(q)\frac{d_u^2(u',\om')d_u^2(u',x)}{d_u^2(u',o)d_u^2(u',\om)}.$$
The factor after
$\ep(q)$
in the right hand side tends to 1 as
$u'\to u$.
Thus
$\ep(q')\to\ep(q)$
as
$u'\to u$,
that is, as
$q'\to q$. 
\end{proof}

The set 
$$V=\set{(q,r)\in\reg^+\cP_7\times\R^4}{$|r-\pr_2\circ M_0(q)|<\ep(q)$},$$
where
$\pr_2:\reg^+\cP_7\times\R^4\to\R^4$
is the projection to the second factor, is the 
$\ep$-neighborhood
of
$M_0(\reg^+\cP_7)$
with variable
$\ep=\ep(q)$
in
$\reg^+\cP_7\times\R^4$.
It follows from Lemma~\ref{lem:ep_continuous} that 
$V$
is open in
$\reg^+\cP_7\times\R^4$.
Thus the set 
$$U_V=\set{M\in\cM}{$M(\reg^+\cP_7)\sub V$}$$
of M\"obius structures is open in the fine topology.
The following is a pertubed version of Proposition~\ref{pro:canonical_vp}.

\begin{pro}\label{pro:pertubed_canonical_vp} Every M\"obius structure
$M\in U_V$
satisfies Increment Axiom~(I), that is,
$U_V\sub\cI$.
\end{pro}

\begin{proof} Given
$M\in U_V$,
for any
$q\in\reg^+\cP_7$, $q=1234567$,
such that 4-tuples
$q_{247}$, $q_{157}$
are 
$M$-harmonic,
that is,
$\crr_2(q_{247})=1=\crr_2(q_{157})$,
we have to show that
$\crr_1(q_{345})>\crr_1(q_{123})$
for 
$M$-cross-ratios.

We assume that
$q=(o,\om,v,\om',o',u,x)$,
and for (semi-)metrics
$d_u\in M$, $d_u^0\in M_0$
with infinitely remote point
$u$
we use notations
$d_u(a,b)=|ab|$, $d_u^0(a,b)=|ab|_0$.
The assumption
$M\in U_V$
implies
$|\crr_2^0(q_{247})-1|<\ep$, $|\crr_2^0(q_{157})-1|<\ep$
for 
$M_0$-cross-ratios,
where
$\ep=\ep(q)$.
Since
$q_{247}=(o,v,o',u)$, $q_{157}=(\om,v,\om',u)$,
we have
$1=\crr_2(q_{247})=\frac{|vo'|\cdot|ou|}{|ov|\cdot|o'u|}=\frac{|vo'|}{|ov|}$,
$1=\crr_2(q_{157})=\frac{|v\om'|\cdot|\om u|}{|\om v|\cdot|\om'u|}=\frac{|v\om'|}{|\om v|}$.
Hence,
\begin{equation}\label{eq:ep_harm_canon}
\left|\frac{|vo'|_0}{|ov|_0}-1\right|<\ep,\quad \left|\frac{|v\om'|_0}{|\om v|_0}-1\right|<\ep. 
\end{equation}
Using that
$|o\om|_0=|ov|_0-|\om v|_0$, $|\om'o'|_0=|vo'|_0-|v\om'|_0$,
because
$M_0$
is canonical, we have
$$|o\om|_0-|\om'o'|_0=|ov|_0-|vo'|_0+|v\om'|_0-|\om v|_0$$
and thus using (\ref{eq:ep_harm_canon}) we obtain
\begin{equation}\label{eq:ep_oom}
-\ep(|ov|_0+|\om v|_0)\le|o\om|_0-|\om'o'|_0\le\ep(|ov|_0+|\om v|_0).
\end{equation}

Similarly, since
$|x\om|_0=|xo|_0+|o\om|_0$, $|xo'|_0=|x\om'|_0+|\om'o'|_0$,
we have
$$\crr_1^0(q_{345})-\crr_1^0(q_{123})=\frac{|x\om|_0}{|xo|_0}-\frac{|xo'|_0}{|x\om'|_0}
  =\frac{|o\om|_0}{|xo|_0}-\frac{|\om'o'|_0}{|x\om'|_0}.$$
Using (\ref{eq:ep_oom}) and that 
$|x\om'|_0-|xo|_0=|o\om'|_0$
we obtain
\begin{equation}\label{eq:diff_cr1_below_0}
\crr_1^0(q_{345})-\crr_1^0(q_{123})\ge\frac{|o\om|_0\cdot|o\om'|_0}{|xo|_0\cdot|x\om'|_0}
  -\ep\frac{|ov|_0+|\om v|_0}{|x\om'|_0}.
\end{equation}
By the assumption
$M\in U_V$,
we have
$|\crr_1(p)-\crr_1^0(p)|<\ep$
for 
$p=q_{345}$
and
$p=q_{123}$.
Hence
$\crr_1(q_{345})-\crr_1(q_{123})\ge\crr_1^0(q_{345})-\crr_1^0(q_{123})-2\ep$.
Thus using (\ref{eq:diff_cr1_below_0}) we obtain
\begin{equation}\label{eq:diff_cr1_below}
\crr_1(q_{345})-\crr_1(q_{123})\ge\frac{|o\om|_0\cdot|o\om'|_0}{|xo|_0\cdot|x\om'|_0}
  -\ep\left(2+\frac{|ov|_0+|\om v|_0}{|x\om'|_0}\right). 
\end{equation}
We have
$o\om\sub o\om'$, $xo\sub x\om'$, $\om v\sub ov\sub x\om'$
in
$X_u$.
Thus
$|o\om|_0<|o\om'|_0$, $|xo|_0<|x\om'|_0$, $|\om v|_0<|ov|_0<|x\om'|_0$,
and hence
$$\frac{|o\om|_0\cdot|o\om'|_0}{|xo|_0\cdot|x\om'|_0}>\frac{|o\om|_0^2}{|x\om'|_0^2},\quad 
  \frac{|ov|_0+|\om v|_0}{|x\om'|_0}<\frac{2|ov|_0}{|x\om'|_0}<2.$$
Therefore
$\crr_1(q_{345})-\crr_1(q_{123})>\frac{|o\om|_0^2}{|x\om'|_0^2}-4\ep=0$. 
\end{proof}

\subsection{Convex M\"obius structures}
\label{subsect:convex_moeb}

We introduce the following Axiom for a M\"obius structure
$M\in\cM$,
which implies convexity of the function
$F_{ab}$.

(C) Convexity: for any 
$q\in\reg\cP_6$
with cyclic order
$\co(q)=123456$
such that 
$\crr_3(q_{46})=\crr_3(q_{26})$
we have 
$$\crr_1(q_{12})>\crr_1(q_{14}).$$
A M\"obius structure
$M\in\cM$
is {\em convex}, if it satisfies Axiom~(C).

Axiom~(C) can be rewritten in the following way. Assume we have
$q=(o',x,y,z,o,\om)\in\reg\cP_6$
written in the cyclic order,
$\co(q)=123456$.
Then
$q_{46}=(o',x,y,o)$, $q_{26}=(o',y,z,o)$,
and the assumption
$\crr_3(q_{46})=\crr_3(q_{26})$
is equivalent to
$\de_{x,y,z}(o)=\de_{x,y,z}(o')$,
where 
$$\de_{x,y,z}(o)=\frac{|yo|^2}{|xo|\cdot|zo|}.$$
Further, we have 
$q_{12}=(y,z,o,\om)$, $q_{14}=(x,y,o,\om)$.
Thus the condition
$\crr_1(q_{12})>\crr_1(q_{14})$
is equivalent to
$\de_{x,y,z}(o)>\de_{x,y,z}(\om)$.

\begin{pro}\label{pro:canonical_convex} The canonical M\"obius structure
$M_0$
on
$X$
is convex.
\end{pro}

\begin{proof} In the metric from
$M_0$
with infinitely remote point
$o'$,
we have 
$\de_{x,y,z}(o')=1$.
Thus we have 
$\de_{x,y,z}(o)=1$
and hence
$|yo|^2=|xo|\cdot|zo|$.
Let
$\si=|o\om|$.
Using that 
$M_0$
is canonical, we have 
$|y\om|=|yo|+\si$, $|x\om|=|xo|+\si$, $|z\om|=|zo|+\si$.
Therefore,
$$\de_{x,y,z}(\om)=\frac{(|yo|+\si)^2}{(|xo|+\si)(|zo|+\si)}=
\frac{1+\al\si+\be\si^2}{1+\ga\si+\be'\si^2},$$
where
$\al=2/|yo|$, $\be=1/|yo|^2$, $\ga=\frac{|xo|+|zo|}{|xo|\cdot|zo|}$,
$\be'=1/(|xo|\cdot|zo|)$.
Since
$|yo|^2=|xo|\cdot|zo|$,
we have
$\be=\be'$,
and thus the inequality
$\de_{x,y,z}(\om)<1$
is equivalent to
$\sqrt{|xo|/|zo|}+\sqrt{|zo|/|xo|}>2$,
which is always true because
$x\neq z$.
\end{proof}

Let
$a=(o,o')$, $b=(\om,\om')\in\ay$
be events in the strong causal relation such that the pairs
$(o,\om')$
and 
$(o',\om)$
separate each other. Using the parametrization
$x\leftrightarrow x_a$
of the arc 
$oo'$
between
$o$, $o'$
that does not contain
$b$
by the timelike line
$h_a$, $x\in oo'$, $x_a\in h_a$,
and the parametrization
$x'\leftrightarrow x_b'$
of the arc 
$\om\om'$
between
$\om$, $\om'$
that does not contain
$a$
by the timelike line
$h_b$,
we consider the function
$F_{ab}:D_{ab}\to\R$,
see sect.~\ref{subsect:hierarchy}, as a function defined on
$h_a\times h_b$, $F_{ab}:h_a\times h_b\to\R$.

\begin{pro}\label{pro:convexity_axiom_convex} Convexity Axiom~(C)
implies that the function
$F_{ab}:h_a\times h_b\to\R$
is strictly convex for any events
$a$, $b\in\ay$
in the strong causal relation.
\end{pro}

\begin{rem}\label{rem:convexity} 1. The convexity of the function
$F_{ab}$
is a precise analog of the convexity of the distance function in 
$\CAT(-1)$ spaces, cp. Proposition~\ref{pro:functional_canonical}.

2. The convexity property depends on a parametrization up to an affine
equivalence. Here, the parametrization of
$D_{ab}$
by 
$h_a\times h_b$
is chosen because
$h_a\times h_b$
is an affine space isomorphic to
$\R\times\R$.
\end{rem}

\begin{proof}[Proof of Proposition~\ref{pro:convexity_axiom_convex}]
As usual, we assume that 
$a=(o,o')$, $b=(\om,\om')\in\ay$
are events in the strong causal relation such that the pairs
$(o,\om')$
and 
$(o',\om)$
separate each other, and 
$e=(o,\om)$, $e'=(o',\om')$.
We show that the increment of the function
$F_{ab}$
strictly increases along any line in
$h_a\times h_b=\R^2$.
To this end, it suffices to show that for any 
$x_a$, $y_a$, $z_a\in h_a$, $x_a<y_a<z_a$,
such that 
$t(x_a,y_a)=t(y_a,z_a)$
we have 
$\De F_{a,b}(z_a,y_a)>\De F_{a,b}(y_a,x_a)$,
where
$$\De F_{a,b}(y_a,x_a)=\frac{1}{2}(t(y_e,x_e')+t(y_{e'},x_{e'}')
                                  -t(x_e,x_e')-t(x_{e'},x_{e'}'))
$$
for some
$x_b'\in h_b$
which is independent of
$x_b'$
(recall that we use here parametrizations
$x\leftrightarrow x_a$ 
and 
$x'\leftrightarrow x_b'$).
Indeed, without loss of generality, we assume that
$q=(o',x,y,z,o,\om)\in\reg\cP_6$
is written in the cyclic order. Then
$t(y_e,x_e')-t(x_e,x_e')=t(y_e,x_e)$,
$t(y_{e'},x_{e'}')-t(x_{e'},x_{e'}')=-t(y_{e'},x_{e'})$,
and thus
$$\De F_{a,b}(y_a,x_a)=\frac{1}{2}(t(y_e,x_e)-t(y_{e'},x_{e'})).$$
The condition
$t(x_a,y_a)=t(y_a,z_a)$
is equivalent to
$\frac{|yo'|\cdot|xo|}{|yo|\cdot|xo'|}=\frac{|zo'|\cdot|yo|}{|zo|\cdot|yo'|}$
for any semi-metric from
$M$,
or which is the same to
$\de_{x,y,z}(o)=\de_{x,y,z}(o')$.
Axiom~(C) implies
$\de_{x,y,z}(o)>\de_{x,y,z}(\om)$.
Since
$$t(z_e,y_e)=\frac{|yo|\cdot|z\om|}{|zo|\cdot|y\om|}\quad\textrm{and}\quad
  t(y_e,x_e)=\frac{|xo|\cdot|y\om|}{|yo|\cdot|x\om|},$$
this is equivalent to
$t(z_e,y_e)>t(y_e,x_e)$.

Applying the same argument to
$q'=(o,z,y,x,o',\om')\in\reg\cP_6$,
we obtain that Axiom~(C) implies
$\de_{x,y,z}(o')>\de_{x,y,z}(\om')$,
which is equivalent to
$t(z_{e'},y_{e'})<t(y_{e'},x_{e'})$.
Therefore,
$\De F_{a,b}(z_a,y_a)>\De F_{a,b}(y_a,x_a)$,
and the strict convexity of the function
$F_{ab}$
follows.
\end{proof}

\begin{rem}\label{rem:axioms_I_C} By Proposition~\ref{pro:convexity_axiom_convex},
Axiom~C implies that the function
$F_{ab}:D_{ab}\to\R$
achieves the infimum at a unique point
$d_0'\in D_{ab}$
for any 
$a$, $b\in\ay$
in the strong causal relation because
$F_{ab}(d)\to\infty$
as
$d$
approaches to the boundary
$\d D_{ab}$
of
$D_{ab}$.
However, in general there is no reason that 
$d_0'=h_a\cap h_b$.
It seems that Axioms~(I) and (C) are independent of each other.
\end{rem}

\section{Appendix 1}
\label{sect:appendix_1}

We show that Gromov hyperbolic spaces from a large class
are boundary continuous, see sect.~\ref{subsect:boundary_continuous}.

\begin{thm}\label{thm:cat0_boundary_contiuous} Every proper Gromov
hyperbolic $\CAT(0)$ space 
$Y$
is boundary continuous.
\end{thm}

For 
$\CAT(-1)$ 
spaces this is established in \cite[Proposition~3.4.2]{BS1}.
Here, we extend this result to
$\CAT(0)$
spaces. A distinction between
$\CAT(-1)$
and 
$\CAT(0)$
cases relevant to arguments is that
$\dist(\ga,\ga')=\inf\set{d(s,s')}{$s\in\ga,s'\in\ga'$}=0$
for asymptotic geodesic rays 
$\ga$, $\ga'$
in the former case, while that distance is only finite in last case. 
This distinction is compensated by the following Lemma.

We use the notation
$o_t(1)$
for a quantity with
$o_t(1)\to 0$
as
$t\to\infty$.

\begin{lem}\label{lem:o(1)_difference} Let
$xyz\sub\R^2$
be a triangle with
$|yz|\le d$
for some fixed
$d>0$
and 
$|xy|,|xz|\ge t$.
Assume 
$\angle_z(x,y)$, $\angle_y(x,z)\ge\pi/2-o_t(1)$.
Then
$||xy|-|xz||=o_t(1)$.  
\end{lem}

\begin{proof} The required estimate follows from the convexity of the distance function on
$\R^2$
and the first variation formula. We leave details to the reader.
\end{proof}

Recall that in a geodesic metric space, the Gromov product is {\em monotone} 
in the following sense, see e.g. \cite[Lemma~2.1.1]{BS1}.

\begin{lem}\label{lem:gromov_product_monotone} Let
$Y$
be a geodesic metric space, 
$xyz\sub Y$
a geodesic triangle. Then for any 
$y'\in xy$, $u\in yz$
we have
$$(y'|z)_x\le (y|z)_x\le\min\{(y|u)_x,(u|z)_x\}.$$
\end{lem}

\begin{proof} The left hand side inequality is equivalent to
$|y'x|-|y'z|\le|yx|-|yz|$,
which follows from the triangle inequality
$|yz|\le|yy'|+|y'z|$
because
$|yx|-|y'x|=|yy'|$.
A similar argument using
$|yz|=|yu|+|uz|$
proves the right hand side inequality.
\end{proof}

All necessarily information about 
$\CAT(0)$
spaces like definition of angles, the triangle inequality for angles,
the comparison of angles, the first variation formula etc used in the proof below can be 
found in \cite{BH}.

\begin{proof}[Proof of Theorem~\ref{thm:cat0_boundary_contiuous}]
Given
$o\in Y$, $\xi$, $\xi'\in\di Y$,
we have to show that for any sequences
$\{x_i\}\in\xi$, $\{x_i'\}\in\xi'$
there is a limit
$\lim_i(x_i|x_i')_o$.
We can assume that
$\xi\neq\xi'$
because otherwise there is nothing to prove.

We use the notation
$\xi=\xi(t)$
for the unit speed parametrization of the geodesic ray 
$o\xi$
with 
$\xi(0)=o$.
By monotonicity of the Gromov product, see Lemma~\ref{lem:gromov_product_monotone}, 
there is a limit
$$a=\lim_{t\to\infty}(\xi(t)|\xi'(t))_o.$$
We have
$a<\infty$
because
$Y$
is hyperbolic and 
$\xi\neq\xi'$,
which implies that the geodesic segment
$\xi(t)\xi'(t)$
stays at uniformly in
$t$
bounded distance from
$o$.
Since
$Y$
is proper, the segments
$\xi(t)\xi'(t)$
subconverge in the compact-open topology as
$t\to\infty$
to a geodesic
$\ga\sub Y$
with the end points
$\xi$, $\xi'$
at infinity.

(1) We fix 
$p\in\ga$
and show that 
$|x_ip|+|px_i'|=|x_ix_i'|+o_i(1)$.
The geodesic segments
$px_i$, $px_i'$
converge to subrays
$p\xi$, $p\xi'\sub\ga$
respectively in the compact-open topology as 
$i\to\infty$.
It follows that the angle
$\angle_p(x_i,x_i')\ge\pi-o_i(1)$.

Let 
$q_i\in x_ix_i'$
be the point closest to
$p$.
By hyperbolicity of 
$Y$
we have
$|pq_i|=\dist(p,x_ix_i')\le d$
for some 
$d>0$
and all 
$i$.
For the triangles
$\De_i=pq_ix_i$, $\De_i'=pq_ix_i'$
we have 
$\angle_{q_i}(p,x_i)$, $\angle_{q_i}(p,x_i')\ge\pi/2$,
and 
$\angle_p(x_i,q_i)+\angle_p(q_i,x_i')\ge\angle_p(x_i,x_i')\ge\pi-o_i(1)$.

Using the comparison of angles for 
$\CAT(0)$
spaces, we see that the comparison triangles
$\wt\De_i=\wt p\wt q_i\wt x_i$, $\wt\De_i'=\wt p\wt q_i\wt x_i'\sub\R^2$
have angles
$\ge\pi/2$
at
$\wt q_i$,
and
$\angle_{\wt p}(\wt x_i,\wt q_i)\ge\angle_p(x_i,q_i)$,
$\angle_{\wt p}(\wt q_i,\wt x_i')\ge\angle_p(q_i,x_i')$.
Thus
$\angle_{\wt p}(\wt x_i,\wt q_i)$, $\angle_{\wt p}(\wt q_i,\wt x_i')<\pi/2$,
and we obtain
$$\pi-o_i(1)\le\angle_{\wt p}(\wt x_i,\wt q_i)+\angle_{\wt p}(\wt q_i,\wt x_i')<\pi.$$
Hence,
$\angle_{\wt p}(\wt x_i,\wt q_i)$, $\angle_{\wt p}(\wt q_i,\wt x_i')\ge\pi/2-o_i(1)$.
Using that 
$|\wt p\wt q_i|\le d$,
we can apply Lemma~\ref{lem:o(1)_difference} and conclude that
$|\wt x_i\wt p|=|\wt x_i\wt q_i|+o_i(1)$,
$|\wt p\wt x_i'|=|\wt q_i\wt x_i'|+o_i(1)$.
Therefore
$|x_ip|+|px_i'|=|x_ix_i'|+o_i(1)$.

(2) By hyperbolicity of
$Y$,
there are points
$u\in o\xi$, $u'\in o\xi'$, $v_t\in\xi(t)\xi'(t)$
with mutual distances bounded above independent of
$t$.
Thus
$\angle_{\xi(t)}(o,\xi'(t))=\angle_{\xi(t)}(o,v_t)=o_t(1)$, 
$\angle_{\xi'(t)}(o,\xi(t))=\angle_{\xi'(t)}(o,v_t)=o_t(1)$,
that is, the segment
$ov_t$
is observed from
$\xi(t)$
and 
$\xi'(t)$
under arbitrarily small angles as
$t\to\infty$.

(3) Let
$\eta(t)$, $\eta'(t)\in\ga$
be points closest to
$\xi(t)$, $\xi'(t)$
respectively. Since the geodesic
$\ga$
is convex as a set in
$Y$,
we have 
$|\eta(t)\eta'(t)|\le|\xi(t)\xi'(t)|$.
Our next goal is to show that 
$|\xi(t)\xi'(t)|\le|\eta(t)\eta'(t)|+o_t(1)$.

Since the geodesic rays
$o\xi$, $p\xi$
are asymptotic, the distance
$\dist(\xi(t),\ga)$
is uniformly bounded above. 
Using convexity of the distance function on
$Y$,
we conclude that
$g(t)=\dist(\xi(t),\ga)$
and similarly
$g'(t)=\dist(\xi'(t),\ga)$
decrease as
$t\to\infty$.
Then for 
$t'>t$
we have 
$g(t')\le g(t)\le|\xi(t)\eta(t')|$
and similarly
$g'(t')\le g'(t)\le|\xi'(t)\eta'(t')|$.
The first variation formula for 
$\CAT(0)$
spaces, see \cite[Corollary~3.6]{BH}, implies that
$\angle_{\xi(t)}(\eta(t),o)$, $\angle_{\xi'(t)}(\eta'(t),o)\ge\pi/2$
for all 
$t>0$.
Combining that with the estimates from (2) for the angles
$\angle_{\xi(t)}(o,\xi'(t))$, $\angle_{\xi'(t)}(o,\xi(t))=o_t(1)$,
we conclude that 
$\angle_{\xi(t)}(\eta(t),\xi'(t))$, $\angle_{\xi'(t)}(\eta'(t),\xi(t))\ge\pi/2-o_t(1)$.
Therefore, all the angles of the quadrilateral
$\eta(t)\xi(t)\xi'(t)\eta'(t)$
are at least
$\pi/2-o_t(1)$.
We also note that
$g(t)=|\xi(t)\eta(t)|$
and 
$g'(t)=|\xi(t)\eta(t)|\le c$
for all 
$t\ge 0$
and some 
$c>0$
independent of
$t$.

Let
$x(t)y(t)u(t)$, $y(t)z(t)u(t)$
be comparison triangles in
$\R^2$
with vertices
$x(t)$, $z(t)$
separated by the common side 
$y(t)u(t)$
for triangles
$\eta(t)\xi(t)\eta'(t)$, $\xi(t)\xi'(t)\eta'(t)$
in 
$Y$
respectively. Using the comparison of angles in
$\CAT(0)$
spaces and the triangle inequality for angles, we obtain that all 
the angles of the quadrilateral
$x(t)y(t)z(t)u(t)\sub\R^2$
are at least
$\pi/2-o_t(1)$.
Since 
$|x(t)y(t)|$, $|z(t)u(t)|\le c$,
we have 
$\angle_{y(t)}(z(t),u(t))$, $\angle_{u(t)}(x(t),y(t))=o_t(1)$.
Thus
$\angle_{y(t)}(x(t),u(t))$, $\angle_{u(t)}(z(t),y(t))\ge\pi/2-o_t(1)$.
By Lemma~\ref{lem:o(1)_difference},
$|y(t)z(t)|$, $|x(t)u(t)|=|y(t)u(t)|+o_t(1)$,
hence
$|\xi(t)\xi'(t)|\le|\eta(t)\eta'(t)|+o_t(1)$.

(4) Now, we show that 
$\al(t)$, $\al'(t)\ge\pi/2-o_t(1)$,
where
$\al(t)=\angle_{\xi(t)}(\eta(t),\xi)$,
$\al'(t)=\angle_{\xi'(t)}(\eta'(t),\xi')$.
For brevity, we only prove this estimate for the angles
$\al(t)$.

By the first variation formula, we have
$|\xi(t+s)\eta(t)|=|\xi(t)\eta(t)|-s\cos\al(t)+o(s)$
for all sufficiently small
$s\ge 0$.
On the other hand, the function
$g=g(t)$
is convex. Thus it has at every point the right derivative
$d_+ g/dt$,
which is non decreasing. It is nonpositive because
$g(t)$
decreases. Thus
$-d_+ g(t)/dt=o_t(1)$.
Using that 
$g(t+s)\le|\xi(t+s)\eta(t)|$
for every
$s\ge 0$,
we obtain
$$g(t)-s\cos\al(t)+o(s)=|\xi(t+s)\eta(t)|\ge g(t+s)\ge g(t)+s\cdot d_+g(t)/dt$$
for all sufficiently small
$s>0$,
hence
$\cos\al(t)\le-d_+g(t)/dt=o_t(1)$,
and therefore
$\al(t)\ge \pi/2-o_t(1)$.

(5) We show that
$|\xi(t)x_i|=|\eta(t)x_i|+o_{t,i}(1)$
for every sufficiently large fixed
$t$,
and similarly
$|\xi'(t)x_i'|=|\eta'(t)x_i'|+o_{t,i}(1)$.
The geodesic segments
$\xi(t)x_i$, $\eta(t)x_i$
converge in the compact-open topology to subrays
$\xi(t)\xi$, $\eta(t)\xi$
respectively as
$i\to\infty$.
Thus
$\angle_{\xi(t)}(\eta(t),x_i)\ge\al(t)-o_i(1)$
and
$\angle_{\eta(t)}(\xi(t),x_i)\ge\be(t)-o_i(1)$,
where
$\be(t)=\angle_{\eta(t)}(\xi(t),\xi)\ge\pi/2$.
Using (4) and the comparison of angles, we obtain that the angles at
$x$, $y$
of the comparison triangle
$xyz\sub\R^2$
for 
$\xi(t)\eta(t)x_i$
are 
$\ge\pi/2-o_{t,i}(1)$.
By Lemma~\ref{lem:o(1)_difference},
$|\xi(t)x_i|=|\eta(t)x_i|+o_{t,i}(1)$.

(6) Since the geodesic segments
$ox_i$
converge to the ray 
$o\xi$,
we have
$|ox_i|=|o\xi(t)|+|\xi(t)x_i|-o_{t,i}(1)$
for every fixed
$t>0$
and all sufficiently large
$i$.
Similarly,
$|px_i|=|p\eta(t)|+|\eta(t)x_i|-o_{t,i}(1)$.
By (5),
$|ox_i|-|px_i|=|o\xi(t)|-|p\eta(t)|+o_{t,i}(1)$.
Using (1), (3) and 
$|\eta(t)p|+|p\eta'(t)|=|\eta(t)\eta'(t)|$,
we finally obtain
$(x_i|x_i')_o=(\xi(t)|\xi'(t))_o+o_{t,i}(1)$.
Hence
$\lim_i(x_i|x_i')_o=a$.
\end{proof}

\begin{cor}\label{cor:zero_gromov_product} In a proper
Gromov hyperbolic
$\CAT(0)$
space
$Y$,
we have
$(\xi|\xi')_o=0$
if and only if
$\angle_o(\xi,\xi')=\pi$
for  
$o\in Y$, $\xi$, $\xi'\in\di Y$.
\end{cor}

\begin{proof} If
$\angle_o(\xi,\xi')=\pi$,
then
$|xo|+|ox'|=|xx'|$,
and 
$(x|x')_o=0$
for every
$x\in o\xi$, $x'\in o\xi'$.
By Theorem~\ref{thm:cat0_boundary_contiuous},
$(\xi|\xi')_o=0$.
 
Conversely, assume that 
$\angle_o(\xi,\xi')<\pi$.
Then for 
$x\in o\xi$, $x'\in o\xi'$
sufficiently close to
$o$,
we have
$|xo|+|ox'|>|xx'|$,
and thus 
$(x|x')_o>0$.
By monotonicity of the Gromov product and Theorem~\ref{thm:cat0_boundary_contiuous},
$(\xi|\xi')_o\ge (x|x')_o>0$.
\end{proof}

\section{Appendix 2}
\label{sect:appendix_2}

\begin{center} {\Large Viktor Schroeder}
 
\end{center}

Here, it will be shown that Axiom~(t6) follows from the other axioms of 
timed causal spaces. That is, we assume Axioms (h1)--(h6) and (t1)--(t5)
but not (t6) and show that (t6) follows. Given an event
$e=(\al,\be)$
we have a reflection
$\rho=\rho_e:S^1\to S^1$
fixing
$\al$, $\be$.
The M\"obius structure
$M$
is obtained in Theorem~\ref{thm:submoeb_is_moeb} without using (t6).
It gives another timelike line structure
$\cH_M$
and hence for 
$e$
another reflection
$\tau=\tau_e:S^1\to S^1$.
Choose
$x$, $y$
in the same component of
$S^1\sm\{\al,\be\}$
in the order
$\al xy\be$.
We use notation
$[\ ,\ ,\ ,\ ]$
for the cross-ratio
$\crr_3$,
$$[x,y,z,u]:=\frac{|xy||zu|}{|xz||yu|}.$$
Then we have
\begin{equation}\label{eq:tau_equality}
[\al,x,\tau(x),\be]=[\al,y,\tau(y),\be]=1. 
\end{equation}
This cross-ratio satisfies the {\em cocycle property}
$$[\al,x,y,\be][\al,y,z,\be]=[\al,x,z,\be]$$
for any 
$x$, $y$, $z$.
Axiom~(t6) is not used in the proof of Lemma~\ref{lem:pro:timed_monotone:axiom_m}.
By that Lemma, the time of the timed causal space is computed in the usual way via
$M$-cross-ratios. Thus we have
$$\ln[\al,x,y,\be]=-t((x,\rho(x)),(y,\rho(y)))=\ln[\al,\rho(x),\rho(y),\be],$$
and we have by the cocycle property and (\ref{eq:tau_equality})
$[\al,x,y,\be]=[\al,\tau(x),\tau(y),\be]$.
Thus
$$[\al,\rho(x),\tau(x),\be][\al,\tau(x),\rho(y),\be]=[\al,\rho(x),\rho(y),\be]$$
equals
$$[\al,\tau(x),\rho(y),\be][\al,\rho(y),\tau(y),\be]=[\al,\tau(x),\tau(y),\be].$$
Thus 
$[\al,\rho(x),\tau(x),\be]$
is constant for 
$x$
in a connected component of
$S^1\sm\{\al,\be\}$.
In order to prove the result, we have to show that 
$[\al,\rho(x),\tau(x),\be]=1$.
Then be monotonicity
$\rho(x)=\tau(x)$,
and we have (t6).

Now
$[\al,\rho(x),\tau(x),\be]=[\al,\rho(x),x,\be]$
since 
$[\al,\tau(x),x,\be]=1$
and, hence, also
\begin{equation}\label{eq:rho_constant}
[\al,x,\rho(x),\be]\quad\textrm{is constant in}\ x. 
\end{equation}

\begin{figure}[htbp]
\centering
\psfrag{1}{$1$}
\psfrag{2}{$2$}
\psfrag{3}{$3$}
\psfrag{4}{$4$}
\psfrag{5}{$5$}
\psfrag{6}{$6$}
\psfrag{7}{$7$}
\psfrag{8}{$8$}
\psfrag{9}{$9$}
\psfrag{10}{$10$}
\includegraphics[width=0.6\columnwidth]{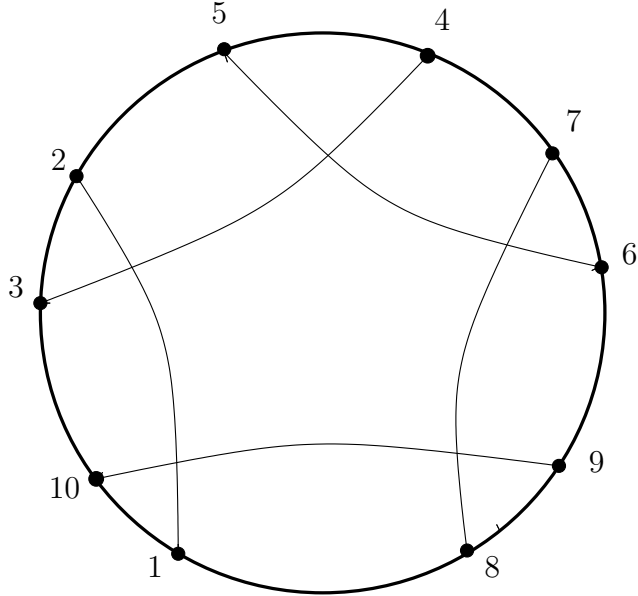}
\caption{Pentagon $P$}\label{fi:fivegone}
\end{figure}

Now, construct a pentagon
$P=x_1x_2,x_2x_3,\dots,x_9x_{10}$
of consecutively ``orthogonal'' timelike lines, i.e.,
$\rho_{x_i,x_{i+1}}(x_{i+2})=x_{i+3}$
for 
$i=1,\dots,9$,
where indexes are taken modulo 10 (existence of
$P$
easily follows from Proposition~\ref{pro:timelike_lines}(b)). Then
(\ref{eq:rho_constant}) implies (we use 
$[i,j,k,l]=[x_i,x_j,x_k,x_l])$
\begin{align*}
 [1,3,4,2]&=[6,3,4,5]\\
          &=[6,8,7,5]\\
          &=[9,8,7,10]\\
          &=[9,1,2,10]\\
          &=[4,1,2,3]=[1,4,3,2]=1/[1,3,4,2],
\end{align*}
hence
$[1,3,4,2]=1$.

\bigskip
\begin{tabbing}

Sergei Buyalo\\

St. Petersburg Dept. of Steklov Math. Institute RAS,\\ 
Fontanka 27, 191023 St. Petersburg, Russia\\
{\tt sbuyalo@pdmi.ras.ru}\\

\end{tabbing}

\end{document}